\documentclass[10pt,leqno,twoside]{amsart}
\usepackage{amsmath,amssymb,amsthm,enumerate}
\usepackage{accents}  
\usepackage[usenames]{color}
\theoremstyle{definition}
 \newtheorem{dfn}{Definition}[section]
 \newtheorem{remark}[dfn]{Remark}
\newtheorem{remarks}[dfn]{Remarks}

\theoremstyle{plain}
 \newtheorem{theorem}[dfn]{Theorem}
 \newtheorem{proposition}[dfn]{Proposition}
 \newtheorem{lemma}[dfn]{Lemma}

\newcommand{\vp}{\varphi}
\newcommand{\R}{\mathbb R}
\newcommand{\N}{\mathbb N}
\newcommand{\C}{\mathbb C}
\newcommand{\ve}{\varepsilon}
\newcommand{\vr}{\varrho}
\newcommand{\cL}{\mathcal L}
\numberwithin{equation}{section}

\title[Well-Posedness of the Primitive Equations]{Global Strong Well-posedness of the Three Dimensional Primitive equations in $L^p$-spaces}

\begin{document}
\author[Matthias Hieber]{Matthias Hieber}
\address{Department of Mathematics, TU Darmstadt, Schlossgartenstr. 7, 64289 Darmstadt, Germany} 
\email{hieber@mathematik.tu-darmstadt.de}
\author[Takahito Kashiwabara]{Takahito Kashiwabara*}
\address{Department of Mathematics, Tokyo Institute of Technology, 2-12-1 Ookayama, Meguro, Tokyo 152-8511, Japan}
\email{tkashiwa@math.titech.ac.jp}
\thanks{* This work was supported by the DFG International Research Training Group 1529 on Mathematical Fluid Dynamics}

\begin{abstract}
In this article, an $L^p$-approach to the primitive equations is developed. In particular, it is shown that the three dimensional primitive equations admit a 
unique, global strong solution for all 
initial data $a \in [X_p,D(A_p)]_{1/p}$ provided $p \in [6/5,\infty)$. To this end, the hydrostatic Stokes operator $A_p$ defined on  $X_p$, the subspace of $L^p$ associated with the 
hydrostatic Helmholtz projection, is introduced and investigated. Choosing $p$ large, one obtains  global well-posedness of the primitive equations for strong solutions for  
initial data $a$ having less differentiability 
properties  than $H^1$, hereby generalizing in particular  a result by Cao and Titi \cite{CaTi07} to the case of non-smooth initial data.    
\end{abstract}

\maketitle

\section{Introduction} \label{sec1}
The primitive equations of the ocean and the  atmosphere are considered to be a fundamental model for many geophysical flows. They are described by a system of equations which are derived from 
the Navier-Stokes or Boussinesq equations for incompressible viscous flows by assuming that  the vertical motion is modeled by the hydrostatic balance. 
This assumption is considered to be corrrect since in these type of models the vertical motion is assumed to be much smaller compared to the horizontal one.     

In this article we are interested in the isothermal situation assuming that the temperature $T$ equals a constant $T_0$. In this case, the primitive equations consist of the 
equations describing  the conservation of momentum and mass of the fluid and are given by  
\begin{eqnarray} \label{eq1.1}
\partial_t v + u\cdot\nabla v - \Delta v + \nabla_H\pi & = & f \quad   \text{ in } \Omega\times(0,T), \nonumber\\
    		\partial_z\pi & = & 0  \quad \text{ in } \Omega\times(0,T), \\
		\mathrm{div}\,u & = & 0  \quad \text{ in } \Omega\times(0,T), \nonumber\\
                             v(0) & = & a. \nonumber
\end{eqnarray}
Here $\Omega = G \times(-h,0)$, where $G=(0,1)^2$  and $h>0$. Moreover, the velocity $u$ of the fluid is described by $u=(v,w)$ with $v=(v_1,v_2)$, and where $v$ and $w$ denote  
the  horizontal and vertical components of $u$, respectively. Furthermore, $\pi$ denotes  the pressure of the fluid (more precisely, $\pi = p + T_0z$, where $p$ is the original 
pressure, $z\in(-h,0)$) and $f$ a given external force. 
The symbol $\nabla_H=(\partial_x,\partial_y)^T$  denotes the horizontal gradient, $\Delta$ the three dimensional Laplacian and 
$\nabla$ and $\mathrm{div}$ the three dimensional gradient and divergence operators. 

The system is complemented by the set of boundary conditions
\begin{equation}\label{eq1.2}
\begin{array}{rll}
\partial_z v & =   0, \quad w=0  &\text{ on }\; \Gamma_u\times(0,T),  \\
	   v & = 0, \quad w=0 & \text{ on }\; \Gamma_b\times(0,T), \\
&	\text{$u$, $\pi$ are periodic}   & \text{ on }\; \Gamma_l\times(0,T).  \\
\end{array}
\end{equation}
Here   $\Gamma_u := G\times\{0\}$, $\Gamma_b := G \times\{-h\}$, $\Gamma_l :=\partial G\times[-h,0]$ 
denote the upper, bottom and  lateral parts of the boundary $\partial\Omega$, respectively.

The full primitive equations were introduced and investigated for the first time by Lions, Temam and Wang in \cite{LTW92a, LTW92b, LTW95}.  
They proved the existence of a global weak solution for this set of equations for initial data $a \in L^2$. Note that the uniqueness question  for these solutions seems to remain an open 
problem until today.

The analysis of the linearized problem goes back to the work of Ziane \cite{Zia95a, Zi97},  who proved $H^2$-regularity of the corresponding resolvent problem. Taking advantage of this result, 
the existence of a local, strong solution with data $a \in H^1$ was proved by Guill\'en-Gonz\'alez, Masmoudi and Rodiguez-Bellido in \cite{GMR01}. 

In 2007, Cao and Titi \cite{CaTi07} proved a breakthrough result for this set of equation which says, roughly speaking,  that there exists a unique, global strong solution to 
the primitive equations for {\it arbitrary} initial data $a\in H^1$. Their proof is based on a priori $H^1$-bounds for the solution, which in turn are 
obtained by $L^\infty(L^6)$ energy estimates.
Note that the boundary conditions on $\Gamma_b\cup\Gamma_l$ considered there are different from the ones we are imposing in \eqref{eq1.2}.
Kukavica and Ziane also considered in \cite{KuZi07, KuZi08} the primitive equations subject to the boundary conditions on $\Gamma_u\cup\Gamma_b$ as in \eqref{eq1.2} and they  
proved global strong well-posedness of the primitive equations with respect to arbitrary $H^1$-data. 
For a different approach see also Kobelkov \cite{Kob07}. The existence of a global attractor for the primitive equations was  proved by Ju \cite{Ju07} and its properties were investigated  
by Chueshov \cite{Chu14}.

Modifications of the primitive equations dealing with either only horizontal viscosity and diffusion or with horizontal or vertical 
eddy diffusivity were recently investigated by Cao and Titi in \cite{CaTi12}, by Cao, Li and Titi in \cite{CLT14b,CLT14c,CLT14a}.      
Here, global well-posedness results are established  for initial data belonging to $H^2$. For recent results concerning the presence of vapor, we refer to the work of 
Coti-Zelati, Huang, Kukavica, Teman and Ziane \cite{CHKTZ14}.

Starting from this situation it is, of course, very   interesting to find spaces with less differentiability properties as $H^1(\Omega)$, which nevertheless guarantee the global 
well-posedness of these equations. We mention here the work of Bresch, Kazhikhov and Lemoine \cite{BKL04} who proved the uniqueness of weak solutions in the two dimensional setting for initial data $a$ with $\partial_za \in L^2$. The  existence of a global, strong solution in the two dimensional setting was proved by the same authors as well as by Petcu, Teman and Ziane in 
\cite[Section 3.4]{PTZ09}. The problem was revisited recently by Kukavica, Pei, Rusin and Ziane  in  \cite{KPRZ14}. The authors show uniqueness of 
weak solutions under the assumption that the initial data are only continuous in the space variables.  
It seems that all of the existing results concerning the well-posedness of the primitive equations are phrased so far within the $L^2$-setting. One reason for this might be 
anisotropic structure of the nonlinear term in the primitive equations compared to the isotropic structure in the situation of  the Navier-Stokes equations.  

In this article we develop an approach to the primitive equations  {\it within the $L^p$-setting for $1<p<\infty$}  and prove the existence of a {\it unique, global strong solution} to the 
primitive equations for initial data $a \in V_{1/p,p}$ for $p \in  [6/5,\infty)$. Here,  $V_{1/p,p}$ denotes the complex interpolation space between our 
ground space $X_p$ and the domain of the hydrostatic Stokes operator,  which are defined precisely in Section 4.
Choosing in particular $p=2$, we note that our space of initial data $V_{1/2,2}$ coincides with the space $V$ introduced by Cao and Titi in \cite{CaTi07} (up to a compatibility condition due to different boundary conditions), see also \cite{CaTi07, GMR01, KuZi07, PTZ09}.
Observe that $V_{1/p,p} \hookrightarrow H^{2/p,p}(\Omega)^2$ for all $p \in (1,\infty)$. Hence, choosing in particular $p\in [6/5,\infty)$ large, one sees that 
our main result extends the existing  well-posedness results for the primitive equations to initial data $a$ having less differentiability properties  than $H^1(\Omega)$. 

The strategy of our approach may be described as follows. In a first step we show that the solution of the linearized equation is governed by an analytic semigroup $T_p$ on the space $X_p$.
Here $X_p$ is defined as the range of the hydrostatic Helmholtz projection $P_p: L^p(\Omega)^2 \to L^p(\Omega)^2$,  which is introduced precisely in Section 4.  This space  can be viewed as 
the analogue of the solenoidal space $L^p_\sigma(\Omega)$, which is very well known in the study of the Navier-Stokes equations. 
The generator of $T_p$, denoted by $-A_p$, is called the {\it hydrostatic Stokes operator}. 
We then rewrite the  primitive equations equivalently in the form 
\begin{equation} \label{eq1.3}
	\left\{
	\begin{aligned}
		v'(t) + A_pv(t) &= P_pf(t) - P_p(v \cdot \nabla_H v + w\partial_z v), \quad  && t>0, \\
		v(0) &= a. &&
	\end{aligned}
	\right.
\end{equation}
Inspired by the Fujita-Kato approach to the Navier-Stokes equations, we consider the integral equation
\begin{equation} \label{eq1.4}
	v(t) = e^{-tA_p}a + \int_0^t e^{-(t-s)A_p} \big( P_pf(s) + F_pv(s) \big)\,ds, \qquad t\geq 0,
\end{equation} 
where $F_pv$ is given  by $F_pv:=- P_p(v \cdot \nabla_H v + w\partial_z v)$. In Section 5 we will prove  the existence of a unique, local, strong solution 
$v \in C([0,T^*];V_{1/p,p} \cap C((0,T^*];V_{\gamma,p})$ to \eqref{eq1.4} for some 
$T^*>0$ and  suitable $\gamma \in (0,1)$ provided $a  \in V_{1/p,p}=[X_p,D(A_p)]_{1/p}$.   
It follows that $v$ is even a local, strong   solution to equation \eqref{eq1.3}, i.e. 
$$
v \in C^1((0,T^*];X_p) \cap C((0,T^*];D(A_p))
$$ 
for all $p \in (1,\infty)$.  Observe that $D(A_p) \hookrightarrow W^{2,p}(\Omega)^2 \hookrightarrow H^1(\Omega)^2$ provided $p \geq 6/5$. 
Hence, one ontains the existence of a unique, global, strong solution to the primitive equations for $a \in  [X_p,D(A_p)]_{1/p}$ for $p \ge 6/5$ provided  
 $\sup_{0\le t \le T} \|v(t)\|_{H^2(\Omega)}$ is bounded by some constant $B=B(\|a\|_{H^2(\Omega)}, T)$ for any $T>0$. Note that our proof for the global $H^2$-bound for $v$ presented in 
Section 6 is based on $L^\infty(L^4)$-estimates for $\tilde v$ and shows that $\|v(t)\|_{H^2(\Omega)}$ is even decaying 
exponentially as $t\to\infty$.

Cao and Titi followed in their celebrated article \cite{CaTi07} a different strategy: they showed that the  local solution constructed there by different means satisfies a global  
$H^1$-bound. Their proof of the $H^1$-bound was based on $L^\infty(L^6)$-estimates for the fluctuating part $\tilde v := v - \bar v$, where $\bar v$ denotes the 
vertical average of $v$ (see \eqref{eq2.02} below).

Finally, let us  compare our choice of the boundary conditions \eqref{eq1.2} with the ones treated in the preceding articles described  above.
The first condition \eqref{eq1.2}$_1$ on $\Gamma_u$ is a slip boundary condition (note, however, that this condition differs from the usual slip boundary condition 
based on the deformation tensor). The authors of  \cite{CaTi07, Ju07, Kob07} suppose that $G$  is a smooth domain in $\R^2$ and impose slip boundary conditions also on $\Gamma_b\cup\Gamma_l$.
In this case, $\overline{\Delta v} = \Delta_H\bar v$, where the bar means the vertical average as before.
If in addition $G$ is a square and $\Gamma_l$ is subject to periodic boundary conditions, we even have the  stronger property that the hydrostatic Helmholtz projection $P_p$ commutes 
with the Laplacian. This implies that the hydrostatic Stokes operator $A_p$ is essentially the Laplacian. This setting is adopted in the analysis of 
\cite{CLT14b, CLT14c, CLT14a, Chu14, KPRZ14} and in \cite[Section 3.6]{PTZ09}.

From the physical point of view it is natural to impose a no-slip condition of the form \eqref{eq1.2}$_2$ on the the bottom boundary $\Gamma_b$; see also \cite[Section 2.1]{PTZ09}.
In this case, the commutator relation described  above does not longer hold true.  
Kukavica and Ziane \cite{KuZi07} supposed that $G$ is a smooth domain and that $\Gamma_l$ is subject to the no-slip condition. They then chose  subintervals of 
$[0,T]$ and integrated the associated differential inequalities on those subintervals.
Their strategy was adopted by \cite{CHKTZ14, KuZi08} and also by  Evans and Gastler \cite{EvGa13}.

In the present paper, we consider ``boundary conditions for the ocean'' on $\Gamma_u\cup\Gamma_b$ of the form \eqref{eq1.2}. 
Since $\partial\Omega$ is non-smooth, we assume $G$ to be a square and introduce periodic boundary conditions on $\Gamma_l$ in order to avoid the corner singularity problem.
Our approach for obtaining global $H^2$-bounds for these or related boundary conditions  is based on $L^\infty(L^4)$-estimates for $\tilde v$ and $L^2(L^2)$-estimates for 
$\nabla_H\pi$.
Adding the estimates for $\tilde v$, $v_z$ and $\nabla_H \pi$ allows us then to apply  a classical Gronwall argument, 
hereby avoiding delicate arguments needed for integrating the associated differential inequalities. 

As already mentioned  above, our approach yields  a smoothing effect for the solution $v$ in the $L^p$-setting.
In fact, we show that $v \in C((0,\infty);W^{2,p}(\Omega)^2)$ for $p \geq 6/5$ and that $v$ even decays exponentially as $t \to \infty$.
        
\section{Preliminaries}
Following Lions, Temam and Wang \cite{LTW92a, LTW92b, LTW95} and Cao and Titi \cite{CaTi07}, we start this section by rewriting the primitive equations given in \eqref{eq1.1} subject to the boundary conditions \eqref{eq1.2}  in an equivalent manner. 
To this end, let us note that the vertical component $w$ of $u$ is determined by the incompressibility condition. More precisely, 
\begin{equation} \label{eq2.1}
	w(x,y,z) = \int_{z}^0 \mathrm{div}_H\,v(x,y,\zeta)\,d\zeta, \qquad (x,y)\in G,\; -h<z<0,
\end{equation}
where we have taken into account the boundary condition  $w=0$ on $\Gamma_u$. The further boundary condition $w=0$ on $\Gamma_b$ gives rise to the constraint
\begin{equation*}
	\mathrm{div}_H\,\bar v = 0 \quad\text{in}\quad G,
\end{equation*}
where $\bar v$ stands for the average of $v$ in the vertical direction, i.e.,
\begin{equation} \label{eq2.02}
	\bar v(x,y) := \frac1h\int_{-h}^0 v(x,y,z)\,dz, \quad (x,y) \in G.
\end{equation}
We observe  that the averaging operator commutes with the tangential differentiation, i.e. we have  $\nabla_H\bar f = \overline{\nabla_Hf}$ for all smooth functions $f$. 
Moreover, according to \eqref{eq1.1}$_2$, $\pi$ can be identified with a function defined only in $G$.

Therefore, problem \eqref{eq1.1}-\eqref{eq1.2} is equivalent to finding a function $v:\Omega\to\mathbb R^2$ and a function $\pi:G\to\mathbb R$ satisfying the set of equations 
\begin{equation} \label{eq2.2}
	\begin{array}{rll}
		\partial_tv + v\cdot\nabla_H v + w\partial_zv - \Delta v + \nabla_H\pi  & = f & \text{ in } \Omega\times(0,T), \\
		w & = \int_z^0 \mathrm{div}_H\,v\,d\zeta & \text{ in } \Omega\times(0,T), \\
		\mathrm{div}_H\,\bar v &= 0 & \text{ in } G\times(0,T), \\
                                    v(0) & = a,&  
	\end{array}
\end{equation}
as well as  the boundary conditions
\begin{equation} \label{eq2.04}
	\begin{array}{rll}
		\partial_{z} v & = 0  & \text{ on } \Gamma_u\times(0,T), \\
		v &= 0 & \text{ on } \Gamma_b\times(0,T), \\
	&	\text{$v$ and $\pi$ are periodic}  & \text{ on } \Gamma_l\times(0,T). \\
	\end{array}
\end{equation}

In the following we will need  a terminology to describe periodic boundary conditions. Let $m \in \N$. We then say that a smooth function 
$f: \overline {\Omega} \to \R$ is {\it periodic of order $m$ on $\Gamma_l$} if
\begin{equation*}
	\frac{\partial^\alpha f}{\partial x^\alpha}(0,y,z) = 
\frac{\partial^\alpha f}{\partial x^\alpha}(1,y,z) \, \mbox{ and } \, \frac{\partial^\alpha f}{\partial y^\alpha}(x,0,z) = \frac{\partial^\alpha f}{\partial y^\alpha}(x,1,z),
\end{equation*}
for all $\alpha=0,\dots,m$.
If the two quantities above are anti-symmetric, then $u$ is said to be {\it anti-periodic}.  In the same way we define the periodicity of order $m$ on $\partial G$ for a 
function defined on $\overline G$. Note that we do not consider any symmetry conditions in the $z$-direction, which is different from the cases studied in \cite{CLT14b, CLT14c, CLT14a, Chu14, KPRZ14} and in \cite[Section 3.6]{PTZ09}.

The Sobolev spaces equipped with periodic boundary conditions in the horizontal directions are defined by
\begin{align*}
	W^{m,p}_{\mathrm{per}}(\Omega) &:= \{f\in W^{m,p}(\Omega) \,|\, \text{$f$ is periodic of order $m-1$ on $\Gamma_l$} \}, \\
	W^{m,p}_{\mathrm{per}}(G) &:= \{f\in W^{m,p}(G) \,|\, \text{$f$ is periodic of order $m-1$ on $\partial G$} \}.
\end{align*}
Of course, we interpret $W^{0,p}_{\mathrm{per}}(G)$ as $L^p(G)$ and verify that 
\begin{align*}
	C^\infty_{\mathrm{per}}(\overline\Omega) &= \{f\in C^\infty(\overline\Omega) \,|\, \text{$f$ is periodic of arbitrary order on $\Gamma_l$} \}, \\
	C^\infty_{\mathrm{per}}(\overline G) &= \{f\in C^\infty(\overline G) \,|\, \text{$f$ is periodic of arbitrary order on $\partial G$} \}
\end{align*}
are dense subspaces of $W^{m,p}_{\mathrm{per}}(\Omega)$ and $W^{m,p}_{\mathrm{per}}(G)$, respectively.

Given a function $f:\overline\Omega \to \R$, we define its periodic extension $Ef$ to $\overline\Omega_1 := G_1\times(-h,0)$, 
where $G_1 := (-1/2,3/2)^2$, by
\begin{equation*}
	Ef(x+j/2,y+k/2,z) := f(x,y,z), \qquad (x,y,z)\in\Omega,\quad k,j\in\{-1,0,1\}.
\end{equation*}
Notice that $Ef$ is well defined in $G_1$ and $Ef \in W^{m,p}(\Omega_1)$ provided  $f\in W^{m,p}_{\mathrm{per}}(\Omega)$. In this case, 
$\|Ef\|_{W^{m,p}(\Omega_1)} = 2^2\|f\|_{W^{m,p}(\Omega)}$. Similarly, we define the periodic extension of a function $f:\overline G \to \R$ and we assign to it the same symbol $Ef$.
Note that $E$ commutes with the vertical averaging operator, i.e., $\overline{Ef} = E\bar u$.

Next,  we introduce a cut-off function $\chi \in C^\infty_0(G_1)$ in such a way that $\chi \equiv1$ in $\overline G$.
If $f\in W^{m,p}_{\mathrm{per}}(\Omega)$, then  $\chi Ef\in W^{m,p}(\Omega_1)$ and  the same relation also holds when $\Omega$ and $\Omega_1$ are replaced by $G$ and $G_1$, respectively.
It follows that
\begin{equation} \label{eq2.3}
	\|\chi Ef\|_{W^{m,p}(\Omega_1)} \le C_\Omega \|f\|_{W^{m,p}(\Omega)}, \qquad \|\chi Ef\|_{W^{m,p}(G_1)} \le C_G \|f\|_{W^{m,p}(G)}
\end{equation}
for some constants $C_\Omega$ and $C_G$. Extending $\chi Ef$ by zero, the extension can be identified with an element of 
$W^{m,p}(\Omega_L)$ or $W^{m,p}(\mathbb R^2)$, respectively,  depending on whether $f$ is defined on  $\Omega$ or $G$. 
Here, $\Omega_L$ denotes the infinite layer $\Omega_L := \mathbb R^2\times(-h,0)$.
We also choose  a second cut-off function $\tilde\chi \in C^\infty_0(-h, 0)$ in such a way that $\int_{-h}^0 \tilde\chi\,dz = h$.

Poincar\'e type inequalities are frequently  used in the subsequent sections. Let us recall from \cite[Section II.5]{Gal11}) that 
\begin{equation}
\begin{array}{rll}
	\|f\|_{L^p(\Omega)} &\le C_p \|\nabla f\|_{L^p(\Omega)}, & f\in W^{1,p}(\Omega),  f=0 \text{ on $\Gamma_u$ on $\Gamma_b$}, \\
	\|f\|_{L^p(G)} &\le C_p \|\nabla_Hf\|_{L^p(G)}, & f\in W^{1,p}(G)\cap L^p_0(G).
\end{array}
\end{equation}
for some constants $C_p$. Here, for an open set $M \subset \R^n\,(n=2,3)$, $L^p_0(M)$ is defined by $L^p_0(M):=\{ v\in L^p(M) \,|\, \int_M v = 0 \}$. 
We also set  $W^{1,p}_0(M) := \{ v\in W^{1,p}(M) \,|\, \gamma v=0 \text{ on }\partial M \}$, where $\gamma v$ denotes the trace of $v$.
With some abuse of notation, $\gamma$ is often omitted in the following. 
Finally, given $f\in L^p(\Omega)$, it follows that  $\bar f\in L^p(G)$ and that 
\begin{equation} \label{eq2.4}
	\|\bar f\|_{L^p(G)} \le h^{-1/p}\|f\|_{L^p(\Omega)}, \quad h>0.
\end{equation}
Some further words about our general notation are in order. We denote by $\mbox{div}_H$ and $\nabla_H$ the horizontal divergence and gradient, i.e. 
$\mbox{div}_H f := \partial_xf + \partial_y f$ and $\nabla_H f := (\partial_xf, \partial_y f)^T$ for all smooth functions $f$. 
The dual space of a complex Banach space $X$ will be  denoted by $X^*$ and the  pairing between $X^*$ and $X$ is written as 
$\left<\cdot, \cdot \right>_{X}$. If $Y$ is also a   Banach space, then $\cL(X,Y)$ denotes the set of all  bounded linear operators from $X$ into $Y$. Given a linear operator $A$ in $X$, we 
denote by $\vr(A)$ its resolvent set in $\C$.

\section{The resolvent problem in the $L^p$-setting }\label{sec3}     
In this section we study the resolvent problem associated with the linearization of   \eqref{eq2.2} within the $L^p$-setting. More precisely, let $\lambda \in 
\Sigma_{\pi-\varepsilon} := \{\lambda\in\mathbb C: |\arg\lambda| < \pi - \ve \}$ for some $\ve \in(0,\pi/2)$ and  $f\in L^p(\Omega)^2$ for some $1<p<\infty$.
Consider the equation 
\begin{equation} \label{eq3.1}
	\begin{array}{rll}
		\lambda v - \Delta v + \nabla_H\pi & = f & \text{ in } \Omega, \\
		\mathrm{div}_H\,\bar v & = 0 & \text{ in } G,
	\end{array}
\end{equation}
subject to  the boundary conditions
\begin{equation} \label{eq3.2}
		\partial_z v = 0 \text{ on } \Gamma_u, \qquad v = 0 \text{ on } \Gamma_b, \qquad \text{$v$ and $\pi$ are periodic} \text{ on } \Gamma_l.
\end{equation}
The functions considered in this section are regarded to be complex-valued.

It is the aim  of this section to establish the following resolvent estimates for equations \eqref{eq3.1} and \eqref{eq3.2}.

\begin{theorem}\label{thm3.1}
Let  $\lambda\in\Sigma_{\pi-\varepsilon}\cup\{0\}$ for some $\varepsilon\in (0,\pi/2)$. Moreover, let   $p\in(1,\infty)$ and $f\in L^p(\Omega)^2$.
Then equations \eqref{eq3.1} and \eqref{eq3.2} admit a unique solution $(v,\pi)\in W^{2,p}_{\mathrm{per}}(\Omega)^2\times W^{1,p}_{\mathrm{per}}(G)\cap L^p_0(G)$ and there exists a  
constant $C>0$  such that
\begin{equation} \label{eq3.3}
|\lambda|\,\|v\|_{L^p(\Omega)} + \|v\|_{W^{2,p}(\Omega)} + \|\pi\|_{W^{1,p}(G)} \le C \|f\|_{L^p(\Omega)}, \quad \lambda\in\Sigma_{\pi-\varepsilon}\cup\{0\}, 
f\in L^p(\Omega)^2.
\end{equation}
\end{theorem}

Let us remark that, by interpolation, we also obtain under the above assumptions an a priori estimate for $\|v\|_{W^{1,p}}$ of the form 
\begin{equation*}
|\lambda|^{1/2} \,\|v\|_{W^{1,p}(\Omega)} \le C \|f\|_{L^p(\Omega)}, \quad \lambda\in\Sigma_{\pi-\varepsilon}\cup\{0\}, f\in L^p(\Omega)^2.
\end{equation*}

We subdivide the proof of Theorem \ref{thm3.1} into two subsections, the first one dealing with the situation of $p=2$ and the second one with the general case $p \in (1,\infty)$.

\subsection{The $L^2$ case}\label{subsec3.1}
We start by deriving   a weak formulation of the problem  \eqref{eq3.1}--\eqref{eq3.2}. To this end, we introduce the function spaces $V$ and $W$ associated with the  
velocity and the pressure of the fluid by 
\begin{equation*}
	V := \{ \varphi \in H^1_{\mathrm{per}}(\Omega)^2: \varphi = 0 \text{ on }\Gamma_b \}, \qquad W:= L^2_0(G).
\end{equation*}
Note that $V$ and $W$ are closed subspaces of $H^1(\Omega)^2$ and $L^2(G)$, respectively. If $(v,\pi)$ is a classical solution of \eqref{eq3.1}--\eqref{eq3.2}, then 
multiplying $\eqref{eq3.1}_1$ and $\eqref{eq3.1}_2$ by test functions $(\varphi,\phi)\in V\times W$ and integrating over $\Omega$, we obtain
\begin{equation} \label{eq3.4}
	\left\{
	\begin{aligned}
		\lambda(v,\vp)_{L^2(\Omega)} + (\nabla v,\nabla \vp)_{L^2(\Omega)} - (p,\mathrm{div}_H\,\bar \vp)_{L^2(G)} &= (f,\vp)_{L^2(\Omega)} \quad && \vp \in V, \\
		 - (\phi,\mathrm{div}_H\,\bar v)_{L^2(G)} &= 0 && \phi\in\Pi,
	\end{aligned}
	\right.
\end{equation}
where $(\cdot,\cdot)_{L^2(G)}$ means the $L^2(G)$ inner product for either scalars, vectors, or matrices. Conversely, if $(v,\pi)$ is smooth  and satisfies \eqref{eq3.4}, then it 
defines  a classical solution of \eqref{eq3.1}--\eqref{eq3.2}.

Assuming $f\in V^*$, we rephrase equation \eqref{eq3.4} as
\begin{equation} \label{eq3.5}
	\left\{
	\begin{aligned}
		a_\lambda(v,\vp) + b(\pi,\vp) &= \left<f,\vp \right>_V, \quad & \vp \in V, \\
		b(\phi,v) &= 0, &  \phi\in W.
	\end{aligned}
	\right.
\end{equation}
Obviously, $a_\lambda :V\times V\to\mathbb C$ and $b: W\times V\to\mathbb C$ are bounded sesquilinear forms. In the following lemma we show the coerciveness of $a$ and the 
inf-sup condition for $b$.

\begin{lemma} \label{lem3.1} 
Let $\varepsilon\in(0,\pi/2)$ and $\lambda\in\Sigma_{\pi-\varepsilon}\cup\{0\}$. \\
a) There exists a constant $C>0$ such that  
$$
|a_\lambda(\vp,\vp)| \ge C (|\lambda|\|\vp\|_{L^2(\Omega)}^2 + \|\vp\|_V^2), \quad \vp \in V.
$$
b) There exists a constant $C=C_{\Omega}>0$ such that 
$$
C \|\phi\|_W \le \sup_{\vp\in V} \frac{|b(\phi,\vp)|}{\|\vp\|_V}, {\quad \phi \in W}.
$$
\end{lemma}

\begin{proof}
a) The  elementary  estimate  $|s\lambda + t| \ge \sin(\varepsilon/2)(s|\lambda| + t)$ for $s,t\ge0$ applied to the form 
$a(v,v) = \lambda\|v\|_{L^2(\Omega)}^2 + \|\nabla v\|_{L^2(\Omega)}^2$ yields,  combined  with Poincar\'e's inequality, the conclusion.

\noindent	
b) It is well-known (see e.g. \cite[p.\ 81]{GiRa86}) that, given $\phi\in W$, there exists $\tilde \vp \in H^1_0(G)^2$ such that 
\begin{equation} \label{eq3.6}
		C(G)\|\phi\|_W \le \frac{|(\phi, \mathrm{div}_H\,\tilde \vp)_{L^2(G)}|}{\|\tilde \vp\|_{H^1(G)}}.
	\end{equation}
We now set $\vp(x,y,z) := \tilde\chi(z)\tilde \vp(x,y)$ where $\tilde\chi$ denotes the cut-off function introduced above in Section 2. It then  follows that $\vp \in V$, 
that $\|\vp \|_V \le C(\Omega) \|\tilde \vp\|_{H^1(G)}$  and that $\bar \vp = \tilde \vp $. In particular, we have 
$\mathrm{div}_H\,\bar \vp = \mathrm{div}_H\,\tilde \vp$, and \eqref{eq3.6} yields the desired estimate.
\end{proof}

The following result follows from the above observations and the Babu\v{s}ka-Brezzi theory on mixed problems (see e.g. \cite[Corollary I.4.1]{GiRa86}).

\begin{proposition} \label{prop3.1}
Let $f\in V^*$. Then  there exists a unique solution $(v,\pi)\in V\times W$ of equation \eqref{eq3.5} and a constant $C>0$ such that 
$$
\|v\|_V + \|\pi\|_W \le C \|f\|_{V^*}.
$$
\end{proposition}

\begin{remark}
Let us remark at this point that the theory presented in \cite{GiRa86} is proved for real Hilbert spaces. However, its extension to the complex case is 
straightforward and can be done without any essential modifications.
\end{remark}

\begin{proof}[Proof of Theorem \ref{thm3.1} for $p=2$]
In order to simplify our notation let us write $C=C_\epsilon$.
Choosing  $\vp =v$ in equation $\eqref{eq3.5}_1$, we obtain from Lemma \ref{lem3.1}a)
	\begin{equation*}
		C|\lambda| \|v\|_{L^2(\Omega)}^2 \le |a(v,v)| \le \|f\|_{L^2(\Omega)}\|v\|_{L^2(\Omega)},
	\end{equation*}
which gives $|\lambda| \|v\|_{L^2(\Omega)} \le C\|f\|_{L^2(\Omega)}$.
	
In order to  show $H^2$-$H^1$ estimates for $v$ and $\pi$, we apply the method of difference quotients. In fact, fix $i=1$ and let $\delta\in\mathbb R\backslash\{0\}$ be such that 
$|\delta| $ is sufficiently small. We define a difference quotient operator $D_\delta$ by $D_\delta f(x,y,z) := \frac{f(x+{\delta},y,z) - f(x,y,z)}\delta$. Using the periodic extension 
operator $E$, we obtain  an integration by parts formula of the form 
$(f, D_{-\delta}(Eg))_{L^2(M)} = (D_\delta(Ef), g)_{L^2(M)}$ where $M$ is $\Omega$ or $G$. Furthermore, if $f\in H^1_{\mathrm{per}}(\Omega)$, then
	\begin{equation*}
		\|D_\delta(\chi Ef)\|_{L^2(\Omega_L)} \le \|\partial_x(\chi Ef)\|_{L^2(\Omega_L)} \le C(\|f\|_{L^2(\Omega)} + \|\partial_xf\|_{L^2(\Omega)}),
	\end{equation*}
	so that $\|D_\delta(Ef)\|_{L^2(\Omega)} \le C\|f\|_{H^1(\Omega)}$.
	
	After  this preparation we choose $D_{-\delta}(E\vp)$ as a test function in $\eqref{eq3.5}_1$ and  deduce
	\begin{equation} \label{eq3.7}
		a_\lambda(D_\delta(Ev), \vp) + b(D_\delta(E\pi), \vp) = (f, D_{-\delta}(E\vp))_{L^2(\Omega)}, \qquad  \vp\in V.
	\end{equation}
	Next, we  take $\vp = D_\delta(Ev)\in V$ in \eqref{eq3.7}.
	By  $\eqref{eq3.5}_2$, the second term on the left hand side above  vanishes and Lemma \ref{lem3.1}a) implies that
	\begin{align*}
		C\|D_\delta(Ev)\|_V^2 \le |a_\lambda(D_\delta(Ev), D_\delta(Ev))| \le \|f\|_{L^2(\Omega)} \|D_{-\delta}D_\delta(Ev)\|_{L^2(\Omega)} \le C\|f\|_{L^2(\Omega)} \|D_\delta(Ev)\|_V,
	\end{align*}
	which imples  $\|D_\delta(Ev)\|_{H^1(\Omega)} \le C\|f\|_{L^2(\Omega)}$.
	Combining Lemma \ref{lem3.1}b) with \eqref{eq3.7}, we obtain
	\begin{equation*}
		C\|D_\delta (E\pi)\|_W \le \sup_{\vp\in V} \frac{|(f,D_{-\delta}(E\vp)) - (\lambda v, D_{-\delta}(E\vp)) - (\nabla D_\delta(Ev), \nabla(E\vp))|}{\|\vp\|_V},
	\end{equation*}
	which is bounded by $C\|f\|_{L^2(\Omega)}$. Here $(\cdot,\cdot)$ means $(\cdot,\cdot)_{L^2(\Omega)}$.
	Letting $h\to0$ yields $\partial_xv\in H^1(\Omega)^2$ and $\partial_x\pi\in L^2(G)$ with norms being bounded by $C\|f\|_{L^2(\Omega)}$.
	
	The same computation as above for $i=2$ leads to $\nabla_Hv\in H^1(\Omega)^2$ and $\nabla_H\pi\in L^2(\Omega)^2$.
	It also follows that $v$ and $\pi$ are periodic on $\Gamma_l$ and $\partial G$ with order $1$ and $0$, respectively.
	Finally, $\eqref{eq3.1}_1$ implies  $-\partial_z^2v = f - \lambda v + \Delta_Hv - \nabla_H\pi\in L^2(\Omega)^2$.
	Hence, $v\in H^2_{\mathrm{per}}(\Omega)^2$ and $\pi\in H^1_{\mathrm{per}}(G)$.
	The proof of Theorem \ref{thm3.1} for the case $p=2$ is complete.
\end{proof}

\subsection{The case  $1<p<\infty$}\label{subsec3.2}
In the following we  extend the results of the previous subsection to the case where $p\in (1,\infty)$.

Following the  strategy introduced in  \cite[pp.\ 281--282]{Zia95a}, we reduce the problem \eqref{eq3.1}--\eqref{eq3.2} to the two-dimensional Stokes equations and the Poisson equation.
In fact, taking the vertical average in equations  \eqref{eq3.1}--\eqref{eq3.2}, we see that
\begin{equation} \label{eq3.8}
	\begin{array}{rll}
		\lambda\bar v - \Delta_H\bar v + \nabla_H\pi & = \bar f - \frac1h\partial_zv|_{\Gamma_b} & \text{ in }\;G, \\
		\mathrm{div}_H\,\bar v & = 0 & \text{ in }\;G, \\
	&	\text{$\bar v$ and $\pi$ are periodic} & \text{ on }\;\partial G,
	\end{array}
\end{equation}
where $\partial_zv|_{\Gamma_b}$ denotes the trace of $\partial_zv$ to $\Gamma_b$.
If $\pi$ is known, the function  $u$ may then be  determined by the equation 
\begin{equation} \label{eq3.9}
	\begin{array}{rll}
		\lambda v - \Delta v &= f - \nabla_H\pi & \text{ in }\;\Omega, \\
		\partial_zv &= 0 & \text{ on }\;\Gamma_u, \\
		v &= 0 & \text{ on }\;\Gamma_b, \\
	&	\text{$v$ is periodic} & \text{ on }\;\Gamma_l.
	\end{array}
\end{equation}
It is easy to see that \eqref{eq3.8}--\eqref{eq3.9} are equivalent to \eqref{eq3.1}--\eqref{eq3.2}.

From now on we eliminate the periodic boundary conditions by focusing on $(\chi Ev,\chi E\pi)$ rather than on $(v,\pi)$. Recalling that $\chi$ is independent of $z$, we verify that 
$(\chi E\bar v,\chi E\pi)$ is satisfying the equations
\begin{equation} \label{eq3.10}
	\begin{array}{rll}
		\lambda (\chi E\bar v) - \Delta_H(\chi E\bar v) + \nabla_H(\chi E\pi) &= \chi E\bar f - \partial_{z}(\chi Ev)|_{\Gamma_{b2}} 
		 - 2\nabla_H\chi\cdot\nabla_H(E\bar v) - (\Delta_H\chi)E\bar v - (\nabla_H\chi)E\pi &  \\
		\mathrm{div}_H(\chi E\bar v) &= \nabla_H\chi\cdot E\bar v &
	\end{array}
\end{equation}
on $G_2$ together with  $\chi E\bar v  = 0 $  on $\partial G_2$, and that $\chi Ev$  is satisfying the equations  
\begin{equation} \label{eq3.11}
	\begin{array}{rll}
		\lambda (\chi Ev) - \Delta(\chi Ev) &= \chi Ef - \chi\nabla_H (E\pi) 	- 2\nabla_H\chi\cdot\nabla_H(Ev) - (\Delta_H\chi)Ev & \text{ in }\;\Omega_L, \\
		\partial_{z}(\chi Ev) & = 0 & \text{ on }\;\{z = 0\}, \\
		\chi Ev & = 0 & \text{ on }\;\{z = -h\}, \\
	\end{array}
\end{equation}
where $G_2$ is a bounded smooth domain in $\mathbb R^2$ containing $\overline G_1$ and $\Gamma_{b2} = \partial G_2\cap\{z = -h\}$.

The above two sets of equations are regarded as resolvent problems for the two-dimensional Stokes equation in $G_2$ and for the Laplacian with mixed boundary condition on $\Omega_L$. 
Resolvent estimates for these two problems are known within the $L^p$-setting. More precisely, the following resolvent estimates hold true.

\begin{lemma} \cite[Theorem 1.2]{FaSo94}, \cite[Lemma 3.3]{Abe06}. \label{lem3.2} 
Let $\lambda\in\Sigma_{\pi-\ve}\cup\{0\}$ for some $\ve >0$ and $1<p<\infty$. \\ 
a) Let  $f\in L^p(G_2)^2$ and $g\in W^{1,p}(G_2)\cap L^p_0(G_2)$. Then  there exists a unique $(v,\pi)\in W^{2,p}(G_2)^2\times W^{1,p}(G_2)\cap L^p_0(G_2)$ satisfying 
	\begin{equation*}
		\lambda v - \Delta_H v + \nabla_H\pi = f \;\text{ in }G_2, \quad \mathrm{div}_H\,v = g\;\text{ in }G_2, \quad v=0\;\text{ on }\partial G_2.
	\end{equation*}
Furthermore, there exists a constant $C>0$ such that 
$$
|\lambda| \|v\|_{L^p(G_2)} + \|v\|_{W^{2,p}(G_2)} + \|\pi\|_{W^{1,p}(G_2)} \le C(\|f\|_{L^p(G_2)} + \|g\|_{W^{1,p}(G_2)} + |\lambda|\|g\|_{\hat W^{-1,p}(G_2)}).
$$
b) For $f\in L^p(\Omega_L)^2$ there exists a unique $u\in W^{2,p}(\Omega_L)^2$ such that
	\begin{equation*}
		\lambda v - \Delta v = f \;\text{ in }\Omega_L, \quad \partial_zv = 0 \;\text{ on }\{z = 0\}, \quad v = 0 \;\text{ on }\{z = -h\}.
	\end{equation*}
Furthermore, there exists a constant $C>0$ such that 
$$
|\lambda| \|v\|_{L^p(\Omega_L)} + \|v\|_{W^{2,p}(\Omega_L)} \le C \|f\|_{L^p(\Omega_L)}.
$$
\end{lemma}

Based on this preparatory work, we are now in the position to start the proof of Theorem \ref{thm3.1} for general $p \in (1,\infty)$. 
We begin with the  uniqueness assertion.

\begin{lemma} \label{prop3.2}
Under the assumptions of Theorem \ref{thm3.1}, let $(v,\pi)\in W^{2,p}_{\mathrm{per}}(\Omega)^2\times W^{1,p}_{\mathrm{per}}(G)\cap L^p_0(G)$ 
be a solution of \eqref{eq3.1}--\eqref{eq3.2} for $f=0$. Then $v=\pi=0$.
\end{lemma}

\begin{proof}
Let us recall from the above Subsection \ref{subsec3.1} that the assertion for the case $p=2$ was already proved there. From this the case $2<p<\infty$ is obvious.
Note that for the case $1<p<2$, it suffices to show that 
$(v,\pi)\in H^2(\Omega)^2\times H^1(G)$. The assertion follows then by the following  bootstrap argument. By assumption, $\chi Ev$ and $\chi E\pi$ 
solve \eqref{eq3.10}--\eqref{eq3.11}. By Sobolev's embedding theorem, equation \eqref{eq2.4} and the trace theorem, each term on the right hand side  of $\eqref{eq3.10}_1$ 
belongs to $L^{p_1}(G)$ where $p_1 = 2p/(3-p)>p$. Moreover, the right hand side  of $\eqref{eq3.10}_2$ belongs to $W^{1,p_1}(G_2)$, again by the embedding theorem.
Hence,  Lemma \ref{lem3.2}a) leads to $\chi E\pi\in W^{1,p_1}(G_2)$. This yields moreover that  the right hand side of $\eqref{eq3.11}_1$ belongs to $L^{p_1}(\Omega_L)^2$.
Therefore, Lemma \ref{lem3.2}b) implies that  $\chi Ev\in W^{2,p_1}(\Omega_L)^2$. By restriction, $(v,\pi)\in W^{2,p_1}(\Omega)^2\times W^{1,p_1}(G)$.
We finally repeat this procedure  and  after finitely many steps, the number of steps depends only on $p$, we obtain  $(v,\pi)\in H^2(\Omega)^2\times H^1(G)$.
\end{proof}

Next, we establish the resolvent bound given in Theorem \ref{thm3.1}, assuming the existence of a  solution.

\begin{proposition} \label{prop3.3}
Under the assumptions of Theorem \ref{thm3.1}, let $(v,\pi)\in W^{2,p}_{\mathrm{per}}(\Omega)^2\times W^{1,p}_{\mathrm{per}}(G)\cap L^p_0(G)$ be a solution of  
\eqref{eq3.1}--\eqref{eq3.2}.  Then estimate \eqref{eq3.3} follows.
\end{proposition}

\begin{proof}
We subdivide the proof of Proposition \ref{prop3.3} into  two steps.\\
{\it Step 1}:  The resolvent estimate given in Lemma \ref{lem3.2}a) applied to equation  \eqref{eq3.10} yields
\begin{eqnarray*}
\|\nabla_H(\chi E\pi)\|_{L^p(G_2)} &\le& C(\|\chi E\bar f\|_{L^p(G_2)} + \|\partial_{z}(\chi Ev)\|_{L^p(\Gamma_{b2})} + \|E\bar v\|_{W^{1,p}(G_1)} + \|E\pi\|_{L^p(G_1)}) \\
&& + C \|\nabla_H\chi\cdot E\bar v\|_{W^{1,p}(G_1)} + C|\lambda| \|\nabla_H\chi\cdot E\bar v\|_{\hat W^{-1,p}(G_2)},
\end{eqnarray*}
where we have used the fact that $\mathrm{supp}\,\chi\subset G_1$. Therefore, by \eqref{eq2.3}, \eqref{eq2.4} and the trace theorem 
	\begin{equation} \label{eq3.12}
		\|\nabla_H(\chi E\pi)\|_{L^p(G_2)} \le C(\|f\|_{L^p(\Omega)} + \|v\|_{W^{1+1/p+\delta,p}(\Omega)} + \|\pi\|_{L^p(G)} + |\lambda| \|v\|_{W^{1,p'}(\Omega)^*}),
	\end{equation}
where $\delta>0$ can be chosen arbitrarily small. In view of the fact that
	\begin{equation*}
		\|(\nabla_H\chi) E\pi\|_{L^p(G_2)} \le C\|\pi\|_{L^p(G)} \le C\|\nabla_H\pi\|_{L^p(G)} \le C\|\nabla_H(\chi E\pi)\|_{L^p(G_2)},
	\end{equation*}
we find that $\|\chi\nabla_H(E\pi)\|_{L^p(G_2)}$ is also bounded by the right hand side  of \eqref{eq3.12}.
	
Next, applying  Lemma  \ref{lem3.2}b) to the equation \eqref{eq3.11} gives
	\begin{align*}
		|\lambda| \|\chi Ev\|_{L^p(\Omega_L)} + \|\chi Ev\|_{W^{2,p}(\Omega_L)} 
		\le\; C(\|\chi Ef\|_{L^p(\Omega_L)} + \|\chi \nabla_H(E\pi)\|_{L^p(G_2)} + \|Ev\|_{W^{1,p}(\Omega_1)}).
	\end{align*}
By the estimates for $\pi$ and $E$, and restricting $\chi Ev$ and $\chi E\pi$ to $\Omega$ and $G$, respectively, we arrive at
	\begin{align}
		|\lambda| \|v\|_{L^p(\Omega)} + \|v\|_{W^{2,p}(\Omega)} + \|\pi\|_{W^{1,p}(G)} 
		\le\; C(\|f\|_{L^p(\Omega)} + \|v\|_{W^{1+1/p+\delta,p}(\Omega)} + \|\pi\|_{L^p(G)} + |\lambda| \|v\|_{W^{1,p'}(\Omega)^*}). \label{eq3.13}
	\end{align}

\noindent	
{\it Step 2}: We prove by a compactness argument that the terms on the left hand  side of \eqref{eq3.13} can indeed be dominated only by $C\|f\|_{L^p(\Omega)}$. 
We argue by contradiction and assume that \eqref{eq3.3} is not true. Then there exist sequences $(v_j,\pi_j)\in W^{2,p}_{\mathrm{per}}(\Omega)^2\times W^{1,p}_{\mathrm{per}}(G)
\cap L^p_0(G)$, $(\lambda_j)\in \Sigma_{\pi-\epsilon}\cup\{0\}$ and $(f_j) \in L^p(\Omega)$ for $j\in\mathbb N$ such that
\begin{equation} \label{eq3.14}
\lambda_jv_j - \Delta v_j + \nabla_H\pi_j = f_j \;\text{ in } \Omega, \quad \mathrm{div}_H\,\bar v_j = 0 \;\text{ in } G, \quad 
\partial_zv_j = 0 \;\text{ on } \Gamma_u, \quad v_j = 0 \;\text{ on } \Gamma_b,
\end{equation}
and
\begin{gather}
\|\lambda_jv_j\|_{L^p(\Omega)} + \|v_j\|_{W^{2,p}(\Omega)} + \|\pi_j\|_{W^{1,p}(G)} = 1, \quad  j\in\mathbb N, \mbox{ and }\label{eq3.15} \\
		  \|f_j\|_{L^p(\Omega)}\to0 \quad\text{as}\quad j\to\infty. \notag
\end{gather}
By \eqref{eq3.15},  there exist  weakly convergent subsequences $(v_j)$ (still denoted by the same symbol) such that
$v_j':=\lambda_jv_j\rightharpoonup v'$, $v_j\rightharpoonup v$, and $\pi_j\rightharpoonup \pi$ for some 
$v'\in L^p(\Omega)^2,\,v\in W^{2,p}_{\mathrm{per}}(\Omega)^2,\,\pi\in W^{1,p}_{\mathrm{per}}(G)$.
	
Observe now that the embeddings $W^{2,p}(\Omega)^2\hookrightarrow W^{1+1/p+\delta,p}(\Omega)^2$, $W^{1,p}(G)\hookrightarrow L^p(G)$ and 
$L^p(\Omega)\hookrightarrow W^{1,p'}(\Omega)^*$ are compact. Hence, the estimate \eqref{eq3.13} yields that $(v_j'),\,(v_j),\,(\pi_j)$ are Cauchy sequences 
in $L^p(\Omega)^2,W^{2,p}(\Omega)^2,W^{1,p}(G)$, respectively, with respect to the strong topology. Consequently, $(v_j'),\,(v_j), \, (\pi_j)$ are strongly convergent, 
which combined with \eqref{eq3.15} implies $\|v'\|_{L^p(\Omega)} + \|v\|_{W^{2,p}(\Omega)} + \|\pi\|_{W^{1,p}(G)} = 1$.
	
On the other hand, considering the limit  $j\to\infty$ in \eqref{eq3.14}, we obtain
\begin{equation} \label{eq3.16}
		v' - \Delta v + \nabla_H\pi = 0 \;\text{ in } \Omega, \quad \mathrm{div}_H\,\bar v = 0 \;\text{ in } G, \quad \partial_zv = 0 \;\text{ on } \Gamma_u, 
\quad u = 0 \;\text{ on } \Gamma_b,
\end{equation}
and also  $\mathrm{div}_H\,\overline{v'} = 0$. We now distinguish two cases. Assume first that the sequence $(\lambda_j)$ is bounded.  By taking a 
subsequence again if necessary, we may assume $\lambda_j\to\lambda\in \overline{\Sigma_{\pi-\epsilon}}$ as $j\to\infty$. Hence $v = \lambda v$. 
By Proposition \ref{prop3.2}, $v=\pi=0$.
Secondly, assume that  $|\lambda_j|\to\infty$. Then $v=0$. Taking the vertical average and applying the horizontal divergence in the first equation of \eqref{eq3.16} leads 
to $\Delta_H\pi = 0$ in $G$. By the unique solvability of this equation, see Proposition \ref{thm4.1} below, we have $\pi=0$ and  thus $v'=0$. We hence achieved a contradiction 
in both  cases, which implies that estimates \eqref{eq3.3} holds true.
\end{proof}

Finally,  we prove the existence of a solution.

\begin{proposition} \label{prop3.4}
Under the assumptions of Theorem \ref{thm3.1}, there exists  $(v,\pi)\in W^{2,p}_{\mathrm{per}}(\Omega)\times W^{1,p}_{\mathrm{per}}(G)\cap L^p_0(G)$ 
satisfying \eqref{eq3.1}--\eqref{eq3.2}.
\end{proposition}

\begin{proof}
We recall from Subsection \ref{subsec3.1} that the case $p=2$ has already been proved there. In the following, we prove the assertion either by a bootstrap argument for $2<p<\infty$ or 
by a density argument for $1<p<2$.
	
Assume, for the time being, that $2<p\le 4$. Then $f\in L^2(\Omega)^2$ and there exists a solution $(v,\pi)\in H^2_{\mathrm{per}}(\Omega)^2\times H^1_{\mathrm{per}}(G)\cap L^2_0(G)$ 
to \eqref{eq3.1}--\eqref{eq3.2}. Hence,  equations \eqref{eq3.10}--\eqref{eq3.11} have a solution as well.
By the Sobolev embedding and the trace theorem, each term on the right hand side  of $\eqref{eq3.10}_1$ belongs to $L^4(G_2)$ (note that the trace $(\partial_{z}\cdot)|_{z = -h}$ 
maps $H^2(G_2\times(-h,0))$ into $L^{4}(G_2)$). We  also see that the right hand side of $\eqref{eq3.10}_2$ lies  in $W^{1,p}(G_2)$.
Applying Lemma \ref{lem3.2}a), we deduce  $\chi E\pi\in W^{1,p}(G_2)$. Thus,  each term on the right hand side  of $\eqref{eq3.11}_1$ lies in $L^p(\Omega)$, 
which combined with Lemma \ref{lem3.2}b) implies $\chi Ev\in W^{2,p}(\Omega_L)^2$. By restriction we obtain $(v,\pi)\in W^{2,p}_{\mathrm{per}}(\Omega)^2\times W^{1,p}_{\mathrm{per}}(G)
\cap L^p_0(G)$.
	
We now repeat the argument for $4<p<\infty$ in order to conclude that the solution constructed in the $L^2$-framework admits  $W^{2,p}$-$W^{1,p}$ regularity.
	
Finally, let  $1<p<2$. Then exists a sequence $(f_j)\subset L^2(\Omega)^2$ such that $f_j\to f$ in $L^p(\Omega)^2$. For each $f_j$ we find a 
solution $(v_j, \pi_j)\in H^2_{\mathrm{per}}(\Omega)^2\times H^1_{\mathrm{per}}(G)\cap L^2_0(G)$ of \eqref{eq3.1}--\eqref{eq3.2}.
By Proposition \ref{prop3.3}, $(v_j, \pi_j, f_j)$ admit an estimate of the form \eqref{eq3.3}. Hence,  $(v_j, \pi_j)$ is bounded in $W^{2,p}(\Omega)^2\times W^{1,p}(G)$.
Extracting a weakly convergent subsequence, we see  that its limit is  a solution of \eqref{eq3.1}--\eqref{eq3.2} for $f$.
\end{proof}

Finally, combining Lemmas \ref{prop3.2} with Propositions \ref{prop3.3} and \ref{prop3.4}, the proof of Theorem \ref{thm3.1} is complete.

\section{The Hydrostatic Stokes operator} \label{sec4}
In this section we introduce the hydrostatic Helmholtz projection and the hydrostatic Stokes operator within the $L^p$-setting. 
They can  be viewed as the analogue of the classical Helmholtz projection and the classical Stokes operator, now, however,  in the situation of the primitive equations.
We will prove in Proposition \ref{propstokesanalytic} that the hydrostatic Stokes operator generates a bounded  analytic semigroup on the subspace $X_p$ of $L^p(\Omega)$. The space $X_p$ 
is strongly related to the hydrostatic Helmholtz projection, which we will introduce in the following.  
The hydrostatic Stokes operator and the associated hydrostatic Stokes semigroup will be of central importance in the construction of a unique, global strong solutions to the 
primitive equations later on.  

As in the case of the classical Helmholtz projection, the existence of the hydrostatic  Helmholtz projection is closely related to the unique solvability of 
the Poisson problem in the weak sense. In our situation, the equation $\Delta_H\pi = \mathrm{div}_H f$ in $G$ subject to  periodic boundary conditions plays an essential role.
We begin with the situation of Dirichlet boundary conditions for domains with  smooth boundaries. Before doing this let us note  that for $1<p,p'<\infty$ with $1/p+1/p'=1$ we have 
\begin{equation*}
	\hat W^{-1,p}(G_2) := \hat W^{1,p'}(G_2)^* \quad \mbox{and} \quad \hat W^{1,p'}(G_2) := \{ f\in L^{p'}_{\mathrm{loc}}(\overline G_2): \nabla_Hf\in L^{p'}(G_2)^2 \}.
\end{equation*}
Since $\hat W^{1,p'}(G_2)$ can be identified with a subspace of $W^{1,p'}(G_2)$ (see e.g. \cite[p.\ 609]{FaSo94}), we see that  $W^{1,p'}(G_2)^*\subset \hat W^{-1,p}(G_2)$.

\begin{lemma} \label{lem4.1}
Let $p\in(1,\infty)$ and $f\in (W^{1,p'}_0(G_2)^2)^*$. Then there exists a unique $\pi\in W^{1,p}_0(G_2)$ such that 
$\left< \nabla_H\pi,\nabla_H\phi \right>_{L^{p'}(G_2)} = \left< f,\phi \right>_{W^{1,p'}_0(G_2)}$ for all $\phi\in W^{1,p'}_0(G_2)$.
Furthermore, there exists a constant $C>0$ such that  $\|\pi\|_{W^{1,p}(G)} \le C_p\|f\|_{W^{1,p'}_0(G_2)^*}$.
\end{lemma}

\begin{proof}
In view of the inf-sup condition
	\begin{equation*}
		C(G_2,p)\|\pi\|_{W^{1,p}(G_2)} \le \sup_{\phi\in W^{1,p'}_0(G_2)} \frac{|\left< \nabla_H\pi,\nabla_H\phi \right>_{L^{p'}(G_2)}|}{\|\phi\|_{W^{1,p'}_0(G_2)}},  
\qquad \pi\in W^{1,p}_0(G_2),
	\end{equation*}
(for a proof see e.g. \cite[Theorem 6.1]{Sim72}), the assertion  results from the generalized Lax-Milgram theorem; see e.g. \cite[Theorem 2.6]{ErGu04}.
\end{proof}

We now turn our attention to the case of the periodic boundary condition.

\begin{proposition} \label{thm4.1}
Let $p\in(1,\infty)$ and $f\in L^p(G)^2$. Then there exists a unique $\pi\in W^{1,p}_{\mathrm{per}}(G)\cap L^p_0(G)$ satisfying 
	\begin{equation} \label{eq4.1}
		\left< \nabla_H\pi, \nabla_H\phi \right>_{L^{p'}(G)} = \left< f, \nabla_H\phi \right>_{L^{p'}(G)}, \quad \phi\in W^{1,p'}_{\mathrm{per}}(G)\cap L^{p'}_0(G).
	\end{equation}
Furthermore, there exists a constant $C>0$ such that 
	\begin{equation} \label{eq4.2}
		\|\pi\|_{W^{1,p}(G)} \le C \|f\|_{L^p(G)}, \quad f\in L^p(G)^2.
	\end{equation}
\end{proposition}

\begin{proof}
The strategy of our proof is similar to the one given in  the proof of Theorem \ref{thm3.1}. Observe first that when $p=2$, the theorem follows immediately from the Lax-Milgram theorem. 

Next, we consider the case where $p \in (1,\infty)$.  We first prove the uniqueness property.  This is obvious 
provided $p\ge2$. When $1<p<2$, it suffices to show that if $\pi\in W^{1,p}_{\mathrm{per}}(G)\cap L^p_0(G)$ is a solution to \eqref{eq4.1} for $f=0$,  then $\pi\in H^1(G)$.
To this end,  we derive a variational equation of which $\chi E\pi$ is a solution. Observe that $E\pi\in W^{1,p}_{\mathrm{per}}(G_1)$ satisfies 
$\left< \nabla_H(E\pi), \nabla_H\phi \right>_{L^{p'}(G_1)} = \left< Ef, \nabla_H\phi \right>_{L^{p'}(G_1)}$ for  $\phi\in W^{1,p'}_{\mathrm{per}}(G_1)$.
By choosing $\chi\phi$ as a test function for  $\phi\in W^{1,p'}_0(G_2)$, we find that $\chi E\pi\in W^{1,p}_0(G_2)$ satisfies
\begin{equation}\label{eq4.3}
		\left< \nabla_H(\chi E\pi), \nabla_H\phi \right>_{L^{p'}(G_2)}  
		= \left< \chi Ef + 2E\pi\,\nabla_H\chi, \nabla_H\phi \right>_{L^{p'}(G_2)} + \left< \nabla_H\chi\cdot Ef + \Delta_H\chi E\pi, \phi \right>_{L^{p'}(G_2)}. 
\end{equation}
Define the functional $F\in W^{1,p'}_0(G_2)^*$ by the right hand side  of \eqref{eq4.3}. Recalling that $f=0$, we obtain  $F\in W^{1,p_1'}(G_2)^*$, where 
$p_1:=2p/(2-p)>p$ due to the Sobolev embedding $W^{1,p}(G_2)\hookrightarrow L^{p_1}(G_2)$. We hence may apply Lemma \ref{lem4.1} to deduce that 
$\chi E\pi\in W^{1,p_1}(G_2)$. By restriction, $\pi\in W^{1,p_1}(G)$. Repeating this procedure finitely many times, we see that  $\pi\in H^1(G)$ and hence, uniqueness is proved.
	
 We now prove \eqref{eq4.2} by assuming that $\pi\in W^{1,p}_{\mathrm{per}}(G)\cap L^p_0(G)$ solves \eqref{eq4.1}. Knowing  that $\chi E\pi$ solves \eqref{eq4.3}, 
Lemma \ref{lem4.1} implies 
	\begin{eqnarray*}
		\|\pi\|_{W^{1,p}(G)} &\le& \|\chi E\pi\|_{W^{1,p}(G_2)} \le C\|F\|_{W^{1,p'}_0(G_2)^*} \\
		&\le &\;C(\|\chi Ef\|_{L^p(G_2)} + \|E\pi\,\nabla_H\chi\|_{L^p(G_2)} + \|\nabla_H\chi\cdot Ef\|_{L^p(G_2)} + \|\Delta_H\chi\, E\pi\|_{L^p(G_2)}) \\
		&\le &\; C(\|f\|_{L^p(G)} + \|\pi\|_{L^p(G)}).
	\end{eqnarray*}
Due to the compactness of the embedding $W^{1,p}(G)\hookrightarrow L^p(G)$ and due to the uniqueness property proved above,  we may  omit the second term on the right hand side of the 
above estimate.
	
Finally, we prove the existence of a solution, addressing the two cases $2<p<\infty$ and $1<p<2$ separately.
For the time being, let $2<p<\infty$. Since $f\in L^2(G)^2$, there exists a solution $\pi\in H^1_{\mathrm{per}}(G)$ to equation \eqref{eq4.1}.  
By Sobolev's theorem, $\pi\in L^{p}(G)$. The functional $F$ given above hence belongs to $W^{1,p'}_0(G_2)^*$. Lemma \ref{lem4.1} implies that   
$\chi E\pi\in W^{1,p}(G_2)$ and  hence $\pi\in W^{1,p}(G)$.
	
Consider finally  the case where  $1<p<2$. By density, there exists a sequence $(f_j)\subset L^2(G)^2$ such that $f_j\to f\in L^p(G)^2$. We associate to each $f_j$ a solution $\pi_j$ of 
\eqref{eq4.1}. Thanks to \eqref{eq4.2}, the sequence $(\pi_j)$ is bounded in $W^{1,p}(G)$. The limit of a weakly converging subsequence constitutes a desired solution.
This completes the proof of Proposition \ref{thm4.1}.
\end{proof}

The above Proposition  \ref{thm4.1} allows us the define the hydrostatic Helmholtz projection   $P_p: L^p(\Omega)^2\to L^p(\Omega)^2$  as follows: 
given $v\in L^p(\Omega)^2$, let  $\pi\in W^{1,p}_{\mathrm{per}}(G)\cap L^p_0(G)$ be the unique solution of equation \eqref{eq4.1} with  $f=\bar v$.
We then set 
\begin{equation}\label{defhelmholtz}
P_pv := v- \nabla_H\pi
\end{equation}
and call  $P_p$ the {\it hydrostatic Helmholtz projection}. It follows from Proposition \ref{thm4.1} that  $P_p^2 = P_p$ and that thus $P_p$ is indeed a projection.

In the following we define the  closed subspace  $X_p$ of $L^p(\Omega)^2$  as $X_p:= \mathrm{Ran} P_p $. This space will play the analogous  role in our investigations of the 
primitive equations as the solenoidal space $L^p_\sigma(\Omega)$ plays in the theory of the 
Navier-Stokes equations. We denote by $\nu_{\partial G}$ the outer unit normal assigned to $\partial G$.

\begin{proposition} \label{prop4.1}
Let $p\in(1,\infty)$. Then the space $X_p$ coincides with the following subsets of $L^p(\Omega)^2$:\\ 
a) $X_p^1:= \{ v\in L^p(G)^2: \left< \bar v, \nabla_H\phi \right>_{L^{p'}(G)} = 0 \quad \mbox{ for all }\phi\in W^{1,p'}_{\mathrm{per}}(G) \}$\\
b) $X_p^2:= \{ v\in L^p(G)^2: \mathrm{div}_H\,\bar v = 0 \text{ and $\bar v\cdot\nu_{\partial G}$ is anti-periodic of order $0$ on $\partial G$} \}$;\\
c) $X_p^3:= \overline{\mathcal V}^{\|\cdot\|_{L^p(\Omega)}}$, \, \mbox{where}
	\begin{equation*}
		\mathcal V = \{ v\in C^\infty_{\mathrm{per}}(\overline\Omega)^2: \mathrm{div}_H\,\bar v = 0 \text{ in }G,\quad \mathrm{supp}\,v\subset \bar G\times(-h, 0) \}.
	\end{equation*}
\end{proposition}

\begin{proof}
a) We prove that  $X_p = X_p^1$. If $v\in X_p^{1}$, then Proposition \ref{thm4.1} leads to $P_pv = v$, so that $v\in X_p$. Conversely, every $v\in X_p$ is represented as 
$v = f - \nabla_H\pi$ where  $f$ and $\pi$ are satisfying \eqref{eq4.1}, which is equivalent to $\left< \bar v, \nabla_H\phi \right>_{L^{p'}(G)} = 0$ 
for all $\phi\in W^{1,p'}_{\mathrm{per}}(G)$. Hence $v\in X_p^1$. \\
b) We prove that $X_p^{1} = X_p^{2}$. To this end, let $v\in X_p^{1}$. Since $C^\infty_0(G)\subset W^{1,p'}_{\mathrm{per}}(G)$, it follows that 
$\mathrm{div}_H\,\bar v = 0$ in the sense of distribution. Thus $\bar v\cdot\nu_{\partial G}$ is a well defined element of $W^{1-1/p',p'}(\partial G)^*$ by the relation
	\begin{equation} \label{eq4.4}
		\left< \bar v\cdot\nu_{\partial G}, \phi \right>_{W^{1-1/p',p'}(\partial G)} = \left< \bar v, \nabla_H\phi \right>_{L^{p'}(G)}, \qquad \phi\in W^{1,p'}_{\mathrm{per}}(G).
	\end{equation}
Let $G_i=G\cap\{x_i=0\}$ for $i=1,2$, where $x_1$ and $x_2$ mean $x$ and $y$ respectively, be the {one} dimensional section of $G$. By extending arbitrary $\phi\in C^\infty_0(G_i)$ constantly along the direction $x_i$, 
we may regard $\phi\in C^\infty_{\mathrm{per}}(\overline G)$. Choosing this $\phi$ in \eqref{eq4.4} we obtain 
	\begin{equation*}
		\left< (\bar v\cdot\nu_{\partial G})|_{\{x_i = {0}\}}, \phi \right>_{C^\infty_0(G_i)} + \left< (\bar v\cdot\nu_{\partial G})|_{\{x_i = {1} \}}, \phi \right>_{C^\infty_0(G_i)} = 0,
	\end{equation*}
which implies that $\bar v\cdot\nu_{\partial G}$ is anti-periodic of order 0 on $\partial G$. Hence $v\in X_p^{2}$.
	
Conversely, let $v\in X_p^{2}$. Note first that   $\mathrm{div}_H\,\bar v = 0$ implies \eqref{eq4.4}. On the other hand, since $\bar v\cdot\nu_{\partial G}$ is anti-periodic 
and $\phi$ in \eqref{eq4.4} is periodic on $\partial G$, it follows that
	\begin{align*}
		\left< \bar v\cdot\nu_{\partial G}, \phi \right>_{W^{1-1/p',p'}(\partial G)} 
		= \sum_{i=1}^2 \left[ \left< \bar v\cdot\nu_{\partial G}, \phi \right>_{W^{1-1/p',p'}(\partial G \cap\{x_i = { 0} \})} + 
\left< \bar v\cdot\nu_{\partial G}, \phi \right>_{W^{1-1/p',p'}(\partial G \cap\{x_i = { 1} \})} \right]
	\end{align*}
equals zero. Therefore, $\left< \bar v, \nabla_H\phi \right>_{L^{p'}(G} = 0$, and thus $v\in X_p^1$. 

Before showing c) we claim that $(X_p^1)^* = X_{p'}^{1}$. In view of the  canonical inclusion $X_{p'}^{1}\subset L^{p'}(G)=L^{p}(G)^*\subset X_p^{1*}$, it suffices to 
show that $X_p^{1*}  \subset X_{p'}^{1}$. Given $F\in X_p^{1*}$, we extend it by the Hahn-Banach theorem to a functional on 
$(L^p(G)^2)^*$, which is represented by a function $f\in L^{p'}(G)^2$. Now Proposition \ref{thm4.1} guarantees that $P_{p'}f \in X_{p'}=X_{p'}^{1}$ is determined 
independently of the way $F$ is extended, which proves the assertion.

We finally prove that $X_p^1 = X_p^{3}$. Since $\mathcal V\subset X_p^{2}=X_p^1$, we see that  $X_p^{3}\subset X_p^1$.
Suppose that $X_p^3 \neq X_p^1$. Then, by the Hahn-Banach theorem, we find $0\neq F\in (X_p^1)^*$ such that $\left< F, v \right>_{X_p^1} = 0$ for all $v\in\mathcal V$.
As already observed above, $F$ is represented by a function $f\in X_{p'}^1$. Choosing $\partial_{z}v$ as a test function, we see that $f$ is independent of $z$.

For arbitrary $\tilde v\in C^\infty_{\mathrm{per},\sigma}(\overline G) := \{ \tilde v\in C^\infty_{\mathrm{per}}(\overline G) \,|\, \mathrm{div}_H\,\tilde v = 0 \}$ we have 
 $v(x,y,z):=\tilde\chi(z)\tilde v(x,y)\in\mathcal V$, so that
	\begin{equation*}
		0 = \left< f, v \right>_{L^p(\Omega)} = \left<f,v\right>_{L^p(G)}, \quad \tilde v\in C^\infty_{\mathrm{per},\sigma}(\overline G).
	\end{equation*}
By de Rham's theorem \cite[Lemma III.1.1]{Gal11}, there exists  $\pi\in W^{1,p'}(G)$ such that $f = \nabla_H\pi$. It follows that 
$\left<\pi, \tilde v\cdot\nu_{\partial G}\right>_{L^p(\partial G)} = 0$ for all $\tilde v\in C^\infty_{\mathrm{per}}(\overline G)$, which implies $\pi\in W^{1,p'}_{\mathrm{per}}(G)$.
However, $\nabla_H\pi=f\in X_{p'}$ results in $\left< \nabla_H\pi, \nabla_H\phi \right>_{L^p(G)} = 0$ for all $\phi\in W^{1,p}_{\mathrm{per}}(G)$, which combined with Proposition 
\ref{thm4.1} leads to $\pi=0$ and thus $f=0$. This contradicts $F\neq0$ and hence $X_p^{3} = X_p^1$.
\end{proof}

The hydrostatic Helmholtz projection $P_p$ defined as in \eqref{defhelmholtz} allows us to define the hydrostatic Stokes operator as follows. In fact, let $1<p<\infty$ and $X_p$
be defined as above. Then the {\it hydrostatic Stokes operator} $A_p$ on $X_p$ is defined as 
\begin{equation}\label{defstokes}
\left\{
\begin{aligned}
A_p v &:= -P_p\Delta v\\
D(A_p)&:= \{ v\in W^{2,p}_{\mathrm{per}}(\Omega)^2: \mathrm{div}_H\,\bar v=0\; \text{ in } G, \quad \partial_{z}v =0\; \text{ on }\Gamma_u, \quad v=0\; \text{ on }\Gamma_b \}.
\end{aligned}
\right.
\end{equation} 

The resolvent estimates for equation \eqref{eq3.1} and \eqref{eq3.2} given in Theorem \ref{thm3.1} yield that $-A_p$ generates a bounded analytic semigroup on $X_p$. 
More precisely, we have the following result.

\begin{proposition}\label{propstokesanalytic}
Let $1<p<\infty$. Then the hydrostatic Stokes operator $-A_p$ generates a bounded analytic $C_0$-semigroup $T_p$ on $X_p$. Moreover, there exist constants $C,\beta >0$ such that
\begin{equation}\label{est4.5}
\|T_p(t)f\|_{X_p} \leq Ce^{-\beta t}\|f\|_{X_p}, \quad  t>0. 
\end{equation}
\end{proposition}

\begin{proof}
Let $\lambda\in\Sigma_{\pi-\ve}\cup\{0\}$ for some $\ve \in (0,\pi/2)$ and $f\in X_p$. Then there exists $u \in D(A_p)$ satisfying $(\lambda + A_p) u = f$ if and only if 
equation \eqref{eq3.1} admits a unique solution $(v,\pi) \in   W^{2,p}_{\mathrm{per}}(\Omega)^2 \times W^{1,p}_{\mathrm{per}}(G)\cap L^p_0(G)$. Hence, 
$\Sigma_{\pi-\varepsilon}\cup\{0\} \subset \varrho(A_p)$ and by Theorem \ref{thm3.1} there exists a constant $C>0$ such that
$$ 
\|\lambda(\lambda + A)^{-1}\|_{\cL(X_p)} \leq C, \quad \lambda\in\Sigma_{\pi-\varepsilon}\cup\{0\}.  
$$
Note further that $A_p$ is densely defined since $\overline{D(A_p)}^{\|\cdot\|_{L^p(\Omega)}} = X_p$, which follows from $\mathcal V\subset D(A_p)$ and Proposition \ref{prop4.1}. 
Furthermore, $A_p$ is closed since $\varrho(A_p) \ne \emptyset$. The assertion thus follows from the generation theorem for analytic semigroups, see e.g. \cite{ABHN11}. 
\end{proof}

\begin{remarks}\label{remap}
a) We remark that due to Theorem \ref{thm3.1}, the graph norm $\|v\|_{D(A_p)} = \|v\|_{X_p} + \|Av\|_{X_p}$ of $D(A_p)$ is equivalent to the $W^{2,p}(\Omega)^2$-norm. \\
b) The theory of analytic semigroup implies that  there exist constants $C,\beta >0$ such that 
\begin{eqnarray}
\|T_p(t)f\|_{D(A_p)} &\leq& C t^{-1}e^{-\beta t}\|f\|_{X_p}, \quad f \in X_p, t>0, \label{est4.6}\\
 \|T_p(t)f\|_{D(A_p)} &\leq& C e^{-\beta t}\|f\|_{D(A_p)}, \quad f \in D(A_p), t>0. \label{est4.7}
\end{eqnarray}
c) The adjoint $A_p^*$ of $A_p$ equals $A_{p'}$. In fact, integrating by parts we obtain $\left< A_{p'}u, v \right>_{X_p} = \left< u, A_pv \right>_{X_p}$ 
for $u\in D(A_{p'})$ and $v\in D(A_p)$. Hence,  $A_{p'}\subset A_p^*$. In order to show the reverse inclusion, let $u\in D(A_p^*)$ and $f\in X_p$. 
By Theorem \ref{thm3.1} we find $\tilde u\in D(A_{p'})$ and $v\in D(A_p)$ such that $A_{p'}\tilde u = A_p^*u\in X_p^*=X_{p'}$ and $A_pv = f$. It follows that
\begin{equation*}
		\left<\tilde u, f\right>_{X_p} = \left<\tilde u, A_pv\right>_{X_p} = \left<A_{p'}\tilde u, v\right>_{X_p} = \left<A_p^*u, v\right>_{X_p} = 
\left<u, A_pv\right>_{X_p} = \left<u, f\right>_{X_p},
\end{equation*}
which implies that $u = \tilde u\in D(A_{p'})$. Hence,  $A_p^*\subset A_{p'}$.
\end{remarks}

The following mapping properties of $T_p(t)=e^{-tA_p}$ related to complex interpolation spaces will be important in the subsequent sections on the nonlinear problem.   
For $0\le\theta\le1$ and $1<p<\infty$, we denote by
\begin{equation} \label{defV}
	V_{\theta,p} := [X_p, D(A_p)]_\theta
\end{equation}
the complex  interpolation space between $X_p$ and $D(A_p)$ of order $\theta$. For more information on interpolation theory, see e.g. \cite{Ama95}.  We then obtain the following result.   

\begin{lemma} \label{thm4.2} 
Let $0\le\theta,\theta_1,\theta_2\le 1$ and assume that  $\theta_1 + \theta_2\le 1$. Then the following assertions hold. \\
a) $V_{\theta,p} \subset H^{2\theta,p}(\Omega)^2$.\\
b) There exists a constant $C>0$ such that 
\begin{equation*}
\|e^{-tA_p}f\|_{V_{\theta_1+\theta_2,p}} \le C t^{-\theta_1}e^{-\beta t} \|f\|_{V_{\theta_2,p}}, \quad f \in V_{\theta_2,p}, t>0.
\end{equation*}
c) $t^{\theta_1}\|e^{-tA_p}f\|_{V_{\theta_1+\theta_2,p}} \to 0$ as $t\to 0$.
\end{lemma}

\begin{proof}
a) This follows from the facts that $X_p\subset L^p(\Omega)^2,\, D(A_p)\subset W^{2,p}(\Omega)^2$ and  $[L^p(\Omega), W^{2,p}(\Omega)]_\theta = H^{2\theta,p}(\Omega)$. \\
b) Interpolating \eqref{est4.6} and \eqref{est4.7} we obtain 
	\begin{equation} \label{eq4.8}
		\|e^{-tA_p}a\|_{D(A_p)} \le C t^{-\theta}e^{-\beta t}\|a\|_{[X_p, D(A_p)]_{1-\theta}}.
	\end{equation}
Interpolating between \eqref{est4.5} and \eqref{eq4.8}, the reiteration theorem implies for $0\le\tau\le1$ 
	\begin{align*}
		\|e^{-tA_p}a\|_{[X_p, D(A_p)]_{\tau}} &\le C t^{-\theta\tau}e^{-\beta t} \|a\|_{[X_p,\, [X_p, D(A_p)]_{1-\theta}]_{\tau}} \\
			&\le C t^{-\theta\tau}e^{-\beta t} \|a\|_{[X_p, D(A_p)]_{(1-\theta)\tau}}.
	\end{align*}
Choosing $\theta$ and $\tau$ such that $\tau=\theta_1+\theta_2$ and $\theta\tau=\theta_1$ yields  the desired estimate. \\
c) For $\tilde a\in D(A_p)$ and $t>0$ we have
	\begin{align*}
		t^{\theta_1}\|e^{-tA_p}a\|_{V_{\theta_1+\theta_2},p} &\le t^{\theta_1}\|e^{-tA_p}(a - \tilde a)\|_{V_{\theta_1+\theta_2},p} + t^{\theta_1}\|e^{-tA_p}\tilde a\|_{V_{\theta_1+\theta_2},p} \\
			&\le C  (\|a - \tilde a\|_{V_{\theta_2},p} + t^{\theta_1}\|a\|_{D(A_p)}).
	\end{align*}
Since $D(A_p)$ is dense in $V_{\theta_2,p}$, the first term on the right hand side above  can be made arbitrarily small and the assertion follows.
\end{proof}

We conclude this section by  considering the special case where $p=2$ and $\theta =1/2$. In this case we are able to characterize the space $V_{\theta,p}$ explicitly. In fact, let 
$$
V_\sigma := \{\vp \in H^1_{\mathrm{per}}(\Omega)^2 : \mathrm{div}_H\,\bar \vp =0 \text{ in } G, \quad \vp =0 \text{ on }\Gamma_b\}. 
$$

\begin{proposition} \label{lem4.5}
Let $V_{\theta,p}$ be defined as in \eqref{defV}. Then 
$$	
V_{1/2,2} = V_\sigma .
$$
\end{proposition}

\begin{proof}
Note that $A_2$ may also be defined by the sesquilinear form $a_\lambda$ given in Section \ref{sec3} for $\lambda=0$. In fact,  $A_2v = f$ for $f\in X_2$ if and only if 
 $a_0(v,\vp) = (f,\vp)_{L^2(\Omega)}$ for all $\vp \in V_\sigma$. Then, by the theory of positive self-adjoint operators on Hilbert spaces (see e.g. \cite[Proposition 4.2]{MiMo08}), 
$V_\sigma = D(A_2^{1/2}) = [X_2, D(A_2)]_{1/2}$.
\end{proof}

\section{Local well-posedness} \label{sec5}
In this section we prove the existence of a unique, mild solution to the system \eqref{eq2.2}-\eqref{eq2.04}. Our approach is inspired by the so called  
Fujita-Kato approach for the Navier-Stokes equations, see e.g. \cite{FuKa64,GiMi85}.

Throughout this section, let $p\in(1,\infty)$. 
We represent the nonlinear terms by
\begin{equation} \label{eq5.1}
	F_pv := -P_p(v\cdot\nabla_Hv + w\partial_zv),
\end{equation}
where $w=w(v)$ is given as in \eqref{eq2.1}. Observe  that $w$ is less regular than $v$ with respect to $(x,y)$, but that $w$ has good regularity properties with respect to $z$.
In order to take into account this {\it anisotropic} nature, we define  for $s,r\ge0$ and $1\le p,q\le\infty$ the function spaces 
\begin{equation*}
	 W^{r,q}_{z}W^{s,p}_{xy} := W^{r,q}((-h, 0); W^{s,p}(G)).
\end{equation*}
Equipped with the norms $\|v\|_{W^{r,q}_{z}W^{s,p}_{xy}} = \big\| \|v(\cdot,z)\|_{W^{s,p}(G)} \big\|_{W^{r,q}(-h,0)}$, they become Banach spaces. We will also use its variants in which 
Sobolev spaces are replaced by Bessel potential spaces.

Taking H\"older's inequality independently with respect to $z$ and $(x,y)$ we obtain
\begin{equation} \label{eq5.2}
	\|fg\|_{L^q_{z}L^p_{xy}} \le \|f\|_{L^{q_1}_{z}L^{p_1}_{xy}} \|g\|_{L^{q_2}_{z}L^{p_2}_{xy}}, \quad 1/p=1/p_1+1/p_2,\; 1/q=1/q_1+1/q_2.
\end{equation}
Embedding relations will  also be performed separately in $z$ and $xy$; in fact we have
\begin{align*}
	W^{r,q}_{z}W^{s,p}_{xy} \hookrightarrow W^{r_1,q_1}_{z}W^{s,p}_{xy} & \qquad \text{provided }\quad W^{r,q}(-h,0)\hookrightarrow W^{r_1,q_1}(-h,0), \\
	W^{r,q}_{z}W^{s,p}_{xy} \hookrightarrow W^{r,q}_{z}W^{s_1,p_1}_{xy} & \qquad \text{provided }\quad W^{s,p}(G)\hookrightarrow W^{s_1,p_1}(G).
\end{align*}
Let us also remark  that $W^{r,p}_{z}W^{s,p}_{xy} \subset W^{r+s,p}(\Omega)$ provided $p=q$. The above  relations hold also in the case of Bessel potential spaces.

After  these preparations we now estimate $F_pv$ in terms of $\|v\|_{V_{\theta,p}}$. 

\begin{lemma} \label{lem5.1}
Let $p \in (1,\infty)$ and $\gamma:=\frac12+\frac{1}{2p}$. Then $F_p$ maps $V_{\gamma,p}$ into $X_p$ and there exists a constant $M>0$ such that\\
a) $\|F_pv\|_{X_p} \le M\|v\|_{V_{\gamma,p}}^2, \, v\in V_{\gamma,p}$, \\
b) $\|F_pv - F_pv'\|_{X_p} \le M (\|v\|_{V_{\gamma,p}} + \|v'\|_{V_{\gamma,p}}) \|v - v'\|_{V_{\gamma,p}}, \, v,v'\in V_{\gamma,p}$. 
\end{lemma}

\begin{proof}
In view of the bilinearity of $v\cdot\nabla_Hv' + w(v)\partial_zv'$ with respect to $v$ and $v'$, assertion b) may be proved similarly as in a). We hence only 
prove a). 
Since $P_p$ is bounded in $X_p$ and $V_{\gamma,p}\subset H^{2\gamma,p}(\Omega)^2$ by Lemma \ref{thm4.2}a), it suffices to bound the $L^p(\Omega)^2$-norms of 
$v\cdot\nabla_Hv$ and $w\partial_zv$ separately by $C\|v\|_{H^{1+1/p,p}(\Omega)}^2$ for some $C>0$. 
	
By H\"older's inequality, $\|v\cdot\nabla_Hv\|_{L^p(\Omega)} \leq C\|v\|_{L^{3p}(\Omega)}\|v\|_{W^{1,3p/2}(\Omega)}$ for some $C>0$. The desired bound follows from the embedding of 
$H^{1+1/p,p}(\Omega)$ into $L^{3p}(\Omega)$ and  $W^{1,3p/2}(\Omega)$. Next, thanks to  \eqref{eq5.2} we obtain  
$\|w\partial_zv\|_{L^p(\Omega)}\le \|w\|_{L^\infty_zL^{2p}_{xy}} \|\partial_zv\|_{L^p_zL^{2p}_{xy}}$. It then follows that
	\begin{align*}
		\|w\|_{L^\infty_zL^{2p}_{xy}} &\le C\|w\|_{W^{1,p}_zL^{2p}_{xy}} \le C\|\partial_zw\|_{L^p_zL^{2p}_{xy}} = C\|\mathrm{div}_H\,v\| _{L^p_zL^{2p}_{xy}} \\
			&\le C\|v\|_{L^p_zW^{1,2p}_{xy}} \le C\|v\|_{L^p_zH^{1+1/p,p}_{xy}} \le C\|v\|_{H^{1+1/p,p}(\Omega)},
	\end{align*}
where we have used the embedding $W^{1,p}(-h,0)\hookrightarrow L^\infty(-h,0)$, Poincar\'e's inequality as well as the embedding  
$H^{1+1/p,p}(G)\hookrightarrow W^{1,2p}(G)$. We also have 
	\begin{equation*}
		\|\partial_zv\|_{L^p_zL^{2p}_{xy}} \le C\|v\|_{W^{1,p}_zL^{2p}_{xy}} = C\|v\|_{H^{1,p}_zL^{2p}_{xy}} \le C\|v\|_{H^{1,p}_zH^{1/p,p}_{xy}} \le C\|v\|_{H^{1+1/p,p}(\Omega)}.
	\end{equation*}
Hence, $\|w\partial_zv\|_{L^p(\Omega)} \le C\|v\|_{H^{1+1/p,p}(\Omega)}^2$ for some $C>0$.
\end{proof}

In the following we prove the existence of a  unique,  mild solution to \eqref{eq2.2}-\eqref{eq2.04}. 
Note first that equations \eqref{eq2.2}-\eqref{eq2.04} can be  rewritten equivalently as 
\begin{equation} \label{eq5.3}
	\left\{
	\begin{aligned}
		 v'(t) + A_pv(t) &= P_pf(t) + F_pv(t), \quad && t>0, \\
		v(0) &= a. &&
	\end{aligned}
	\right.
\end{equation}
Let  $T>0$ and $\delta:=1/p$. Then  $v \in C([0,T]; V_{\delta,p})$ is called a mild solution to equation \eqref{eq5.3} provided $v$ satisfies the integral equation
\begin{equation} \label{eq5.4}
	v(t) = e^{-tA_p}a + \int_0^t e^{-(t-s)A_p} \big( P_pf(s) + F_pv(s) \big)\,ds, \qquad t\ge0.
\end{equation}
In order to formulate the the main result  of this section, we define  for $T>0$, $\delta=1/p$ and $\gamma := \frac12 + \frac{1}{2p}$ the space $\mathcal S_T$ as  
\begin{equation*}
	\mathcal S_T := \{ v\in C([0,T]; V_{\delta,p})\cap C((0,T]; V_{\gamma,p}): \|v(t)\|_{V_{\gamma,p}}=o(t^{\gamma-1}) \text{ as }t\to0 \}.
\end{equation*}
When equipped with the norm
\begin{equation*}
	\|v\|_{\mathcal S_T} := \sup_{0\le s\le T}\|v(s)\|_{V_{\delta,p}} + \sup_{0\le s\le T} s^{1-\gamma}\|v(s)\|_{V_{\gamma,p}},
\end{equation*}
the space $\mathcal S_T$ becomes a Banach space. 

Our local well posedness result reads as follows. 

\begin{proposition} \label{thm5.1}
Let $T>0$, $\delta = 1/p$ and $\gamma = \frac12 + \frac{1}{2p}$ and assume that $a \in V_{\delta,p}$ and $P_pf\in C((0,T];X_p)$ with $\|P_pf(t)\|_{X_p} = o(t^{2\gamma - 2})$ as $t\to0$.
Then there exist  $T^*>0$ and a unique mild solution $v\in \mathcal S_{T^*}$ to  \eqref{eq5.3}.
If in addition $\|a\|_{V_{\delta,p}} + \sup_{0\le s\le T}s^{2-2\gamma}\|P_pf(s)\|_{X_{p}}$ is sufficiently small, then $T^*=T$.
\end{proposition}

\begin{proof} We  subdivide our proof into five steps as follows. \\
{\it Step 1}: Consider an  approximating sequence $v_m\in\mathcal S_T\,(m=0,1,\dots)$ which is defined  by
	\begin{equation} \label{eq5.5}
		v_0(t) := e^{-tA_p}a + \int_0^t e^{-(t-s)A_p}P_pf(s)\,ds, \quad v_{m+1}(t) := v_0(t) + \int_0^t e^{-(t-s)A_p}F_pv_m(s)\,ds, \quad t>0.
	\end{equation}
In order to simpify our notation we set $V_\gamma:= V_{\gamma,p}$. We verify first that  $v_m$ is well defined in $\mathcal S_T$ by noting that  
	\begin{align*}
		\|v_0(t)\|_{V_\gamma} &\le \|e^{-tA_p}a\|_{V_{\gamma}} + \int_0^t \|e^{-(t-s)A_p}\|_{\mathcal L(V_0,V_\gamma)} s^{2\gamma-2}s^{2-2\gamma}\|P_pf(s)\|_{X_p}\,ds \\
			&\le Ct^{\gamma-1}\|a\|_{V_\delta} + C\int_0^t (t-s)^{-\gamma}s^{2\gamma-2}\,ds \sup_{0\le s\le t}(s^{2-2\gamma}\|P_pf(s)\|_{X_p}).
	\end{align*}
Here, we used Lemma \ref{thm4.2}b). Hence 
$$
t^{1-\gamma}\|v_0(t)\|_{V_\gamma} \leq C\|a\|_{V_\delta} + C B(\gamma,2-2\gamma) \sup_{0\le s\le t}(s^{2-2\gamma}\|P_pf(s)\|_{X_p}), \quad t\in(0,T),
$$ 
where $B(\cdot,\cdot)$ denotes the Beta function. The fact that $t^{1-\gamma}\|v_0(t)\|_{V_\gamma}\to 0$ as $t\to 0$ follows from Lemma \ref{thm4.2}c).
A similar computation combined with Lemma \ref{lem5.1}a) gives, 
\begin{equation*}
t^{1-\gamma} \|v_{m+1}(t)\|_{V_\gamma} \le t^{1-\gamma}\|v_0(t)\|_{V_\gamma} + CB(\gamma,2-2\gamma)M \sup_{0\le s\le t} (s^{1-\gamma}\|v_m(s)\|_{V_\gamma})^2, \quad t\in(0,T).
\end{equation*}
By induction, we then see that $v_m\in\mathcal S_T$ for all $m\ge0$. \\
{\it Step 2}: Setting  $k_m(t):=\sup_{0\le s\le t} s^{1-\gamma}\|v_m(s)\|_{V_\gamma}$ and $C_1:=CB(\gamma,2-2\gamma)M$, we deduce that $k_{m+1}(t)\le k_0(t) + C_1k_m(t)^2$ 
for $t>0$ and with $k_m(0)=0$ for $m\ge0$. This quadratic inequality implies for $0<t<T$,
	\begin{equation*}
		\text{if}\quad k_0(t)<1/(4C_1), \quad\text{then}\quad k_m(t) < K(t):=(1 - \sqrt{1 - 4C_1k_0(t)})/(2C_1) < 1/(2C_1).
	\end{equation*}
The assumption of this statement is satisfied provided one of the following assertions are true:
	
	(1) $T$ is sufficiently small (note that $k_0(t)$ is continuous and that $k_0(0)=0$);
	
	(2) $\|a\|_{V_\delta} + \sup_{0\le s\le T}s^{2-2\gamma}\|P_pf(s)\|_{X_p}$ is sufficiently small.
	
\noindent 
Note that the cases (1) and (2) will lead to the local and global existence, respectively. In the following, we investigate the case (1) and choose $T=T^*$ sufficiently small.\\
{\it Step 3}: 
Setting $u_m := v_{m+1} - v_m$, we estimate $\|u_m\|_{V_\gamma}$ by using Lemma \ref{lem5.1}b). We obtain 
	\begin{equation*}
		\sup_{0\le s\le T^*} s^{1-\gamma}\|u_{m+1}(s)\|_{V_\gamma} \le 2C_1K(T^*) \sup_{0\le s\le T^*} s^{1-\gamma}\|u_m(s)\|_{V_\gamma},  \qquad  m\ge0.
	\end{equation*}
Since $2C_1K(T^*)<1$, we see that  
$$
v(t) := v_0(t) + \sum_{m=0}^\infty u_m(t)
$$
exists in $C((0,T^*]; V_{\gamma})$ as a uniform convergence limit. Further, since $K(0)=0$ it follows that 
$\|v(t)\|_{V_\gamma}=o(t^{\gamma-1})$ as $t\to0$. Since $u_m(0)=0$ for $m\ge0$, we also obtain  $v\in C([0,T^*]; V_\delta)$. Consequently, $v\in\mathcal S_{T^*}$.\\
{\it Step 4}: 
By Lemma \ref{lem5.1}{b)}, $F_pv_m(t)\to F_pv(t)$ in $X_p$ as $m\to\infty$ for $t>0$. Moreover, $\|F_pv_m(t)\|_{X_p}$ is bounded by $MK(T^*)^2 t^{2\gamma-2}$, which is integrable on $(0,T^*)$.
Lebesgue's convergence theorem enables us to take the limit in \eqref{eq5.5}, which implies that $v(t)$ is a mild solution to equation \eqref{eq5.3}. \\
{\it Step 5}:  In order to prove the uniqueness of mild solutions, let $v$ and $v'$ be two mild solutions in $\mathcal S_T$ and let $u=v-v'$. Setting 
$\tilde K(t) := \max\{ \sup_{0\le s\le t} s^{1-\gamma}\|v(s)\|_{V_\gamma}, \sup_{0\le s\le t} s^{1-\gamma}\|v'(s)\|_{V_\gamma} \}$ we see that  $\tilde K(0)=0$.
Similarly as in Step 3, we see that $\sup_{0\le s\le t} s^{1-\gamma}\|u(s)\|_{V_\gamma}$ is bounded by $2C_1\tilde K(t) \sup_{0\le s\le t} s^{1-\gamma}\|u(s)\|_{V_\gamma}$ for all $t\in(0,T]$.
Choosing $\tilde T$ such that $2C_1\tilde K(\tilde T)<1$, we obtain $u\equiv0$ on $[0,\tilde T]$. Repeating  this argument yields $u\equiv0$ on $[0,T]$.
\end{proof}

\begin{remarks} \label{rem5.1}
a) The assertion concerning  the global existence of a unique, mild solution can obtained without essential modifications. \\
b) Consider the special case where $p=2$. Then  the regularity condition $a\in V_{1/2,2}$ required for the initial data is characterized by Proposition \ref{lem4.5}.
Note that this condition coincides precisely with the one given in \cite[Theorem 1.2]{GMR01}). Their approach, however, is based on the Galerkin method.  
\end{remarks}

\begin{remark}
Let us clarify the dependency of $T^*$, i.e.\ the length of the existing time of the solution constructed above, on the initial data $a$ for the case $f\equiv0$.
In this case, $T^*$ is chosen in such a way that $k_0(T^*) = \sup_{0\le s\le T^*} s^{1 - \gamma}\|e^{-tA_p}a\|_{V_{\gamma, p}} < 1/(4C_1)$,
whereas $k_0(t)$ is estimated by the use of Lemma 4.6b) as $k_0(t) \le Ct^{\min\{1-\gamma, \ve \}} \|a\|_{V_{\delta + \ve}, p} $ for all $t>0$, 
provided that $a\in V_{\delta + \ve, p}$ with $0\le\ve \le1-\delta$.
Therefore, if $\ve > 0$, then we may set
\begin{equation*}
	T^* = \frac12 \left( \frac1{4CC_1 \|a\|_{V_{\delta + \ve}, p}} \right)^{\max\{1/(1-\gamma), 1/\ve\}},
\end{equation*}
which depends only on the $V_{\delta + \ve, p}$-norm of the initial data.
Note, however, that we cannot assume  $\ve = 0$ above. In fact, for $a\in V_{\delta, p}$ the dependency of $T^*$ on $a$ cannot be controlled merely by the $V_{\delta,p }$-norm of $a$.
\end{remark}

In the following we show  that the mild solution to \eqref{eq5.3}  constructed above  is in fact  a strong solution. 
For proving this, we make use of the following assertions.

\begin{lemma} \label{lem5.2}
a) Let $\theta_1,\theta_2\ge0$ such that $\theta_1+\theta_2\le1$. Then there exists a constant $C>0$ such that 
$\|I - e^{-tA_p}\|_{\mathcal L(V_{\theta_1+\theta_2, p}, V_{\theta_1, p})} \le C t^{\theta_2}, \quad t\ge0$.\\ 
b) For $a\in X_p$ set $z(t):=e^{-tA_p}a$.  Then there exists a constant $C>0$ such that  
$\|z(t+s) - z(t)\|_{V_{\tau, p}} \le C t^{-1+\tau} s^{1-\tau}$ for all  $\tau\in[0,1]$ and all $s\geq 0$. 
\end{lemma}

\begin{proof}
a) By \eqref{est4.7} and \eqref{est4.5}, we obtain $\|I - e^{-tA_{p}}\|_{\mathcal L(V_{\theta,p},V_{\theta,p})}\le C$ for $t\geq 0$ and $\theta=0,1$.
In addition, since $I - e^{-tA_{p}} = -A \int_0^t e^{-sA_{p}}\,ds$, we have  $\|I - e^{-tA_{p}}\|_{\mathcal L(V_{1,p},V_{0,p})}\le C t$ for $t\geq 0$.
Interpolating these three estimates together with the reiteration theorem yields the assertion. \\
b) This follows from the theory of analytic semigroups by  assertion a).
\end{proof}

We now collect mapping properties of the convolution integral 
$$
H(t):= \int_0^t e^{-(t-s)A_p}f(s)\,ds,
$$
where $f\in C((0,T];X_p)$ satisfies certain assumptions as $t\to0$ and $t\to\infty$. For related results, see e.g. \cite[Lemmas 3.4 and 3.5]{His91}.

\begin{lemma} \label{lem5.4}
Let $\kappa\ge0$, $\tau\in(0,1)$ and $f\in C((0,T]; X_p)$. \\
a) Assume that $\|f(t)\|_{X_p}\le Ct^{-\kappa}$ for all $t\in(0,T]$. Then there exists $\ve >0$ and $\tilde C > 0$ such that 
$$
\|H(t+s) - H(t)\|_{V_\tau} \le C {\tilde C} 
\max\{t^{\ve-\kappa}s^{1-\tau-\ve},t^{-\kappa}s^{1-\tau}\}, \quad s \in [0,T-t]. 
$$
b) Assume that $f\in C^\theta((0,T];X_p)$ and that $\|f(t)\|_{X_p} \le L_1 t^{-\kappa}$  for $t\in(0,T]$ as well as \\
$\|f(t+s) - f(t)\|_{X_p} \le L_2t^{-\tau}s^\theta$ for $ t\in(0,T]$ and $s \in [0,T-t]$. 
Then there exists a constant $c>0$ such that 
 $$
\|\partial_t H(t)\|_{X_p} + \|H(t)\|_{D(A_p)} \le c t^{-c}(L_1+L_2e^{-ct}), \quad t>0.
$$
\end{lemma}

\begin{proof}
a) Observe that 
\begin{eqnarray*}
\|H(t+s) - H(t)\|_{V_\tau} &\le \int_0^t \|e^{-sA_p} - I\|_{\mathcal L(V_{1-\epsilon}, V_\tau)} \|e^{-(t-s)A_p}\|_{\mathcal L(V_0, V_{1-\epsilon})} \|f(s)\|_{V_0}\,ds \\
		&  + \int_t^{t+s} \|e^{-(t+s-\sigma)A_p}\|_{\mathcal L(V_0, V_\tau)} \|f(\sigma)\|_{V_0}\,d\sigma.
\end{eqnarray*}
Combining this representation with Lemmas \ref{thm4.2}b) and \ref{lem5.2} yields the assertion. \\
b) Observing that
	\begin{equation*}
		A_pH(t) = (I - e^{-tA_p/2})f(t) + \int_0^{t/2} A_pe^{-(t-s)A_p}f(s)\,ds + \int_{t/2}^t A_pe^{-(t-s)A_p}(f(s) - f(t))\,ds,
	\end{equation*}
the assertion follows from  the estimates  \eqref{est4.6}  and \eqref{est4.7} given in Remark \ref{remap}b). 
\end{proof}

\begin{remark} \label{rem:maximal Holder regularity}
Given  the situation of Lemma \ref{lem5.4} b), we obtain  maximal H\"older regularity of $v$, i.e.\ $v\in C^{1,\theta}((0,T]; X_p)\cap C^\theta((0,T]; D(A_p))$; see 
\cite[Thm 4.3.5]{Paz83}.
\end{remark}

\begin{proposition} \label{thm5.2}
Let $f\in C^\theta((0,T]; X_p)$ with $\theta\in(0,1)$. Then the mild solution $v$ to \eqref{eq5.4} given in Proposition \ref{thm5.1} is a strong solution. More precisely, 
$v \in C^{1,\mu}((0,T]; X_p)\cap C^{\mu}((0,T]; D(A_p))$ satisfies \eqref{eq5.3} for all $t\in(0,T]$ and where $\mu=\min\{\theta, 1-\gamma-\epsilon\}$ and $\epsilon>0$ can be chosen 
arbitrarily small.
\end{proposition}

\begin{proof}
By the proof of Proposition \ref{thm5.1}, $\|v(t)\|_{V_{\gamma,p}} \le Ct^{-C}$ for $t\in(0,T]$. Further, by Lemma \ref{lem5.1}a), $\|F_{p}v(t)\|_{X_p}\le Ct^{-C}$ for 
$t \in (0,T]$. Thus,  Lemmas \ref{lem5.2} and \ref{lem5.4} lead to $v\in C^{1-\gamma-\ve}((0,T]; V_{\gamma,p})$ for some $\varepsilon>0$.
Now Lemma \ref{lem5.1}b) yields  $F_pv\in C^{1-\gamma-\ve}((0,T]; X_{p})$, which combined with Lemma \ref{lem5.2} and Remark \ref{rem:maximal Holder regularity} implies that 
$v\in C^{1,\mu}((0,T]; X_{p})\cap C^\mu((0,T]; D(A_{p}))$. The fact that \eqref{eq5.3}$_1$ holds may be confirmed by a direct computation.
\end{proof}

\begin{remark}
In view of the fact that $\partial_tP_p=P_p\partial_t$, it is now immediate to recover the pressure $\pi$ from equation \eqref{eq5.3}$_1$.
We thus constructed a unique, local solution to equation \eqref{eq2.2}, which as already remarked in Section 2, is the solution to the 
original problem \eqref{eq1.1}--\eqref{eq1.2}.
\end{remark}

\section{$H^{2}$- a priori bounds and global well-posedness for $p \in [\frac65, \infty)$} \label{sec6}

We are now in the position to state the main result of this article. 

\begin{theorem}\label{mainthm}
Let $p\in [6/5,\infty)$, $a\in V_{1/p,p}$ and  $f\equiv0$.
Then there exists a unique, strong global solution $(v,\pi)$ to \eqref{eq2.2}--\eqref{eq2.04} within the regularity class 
$$
v\in C^1((0,\infty); L^p(\Omega)^2)\cap C((0,\infty); W^{2,p}(\Omega)^2), \qquad \pi\in C((0,\infty); W^{1,p}(G)\cap L^p_0(G)).
$$
Moreover, the solution $(v,\pi)$ decays exponentially, i.e. there exist constants $M,c,\tilde c>0$ such that 
\begin{equation} \label{eq:exponential decay}
\|\partial_tv(t)\|_{L^p(\Omega)} + \|v(t)\|_{W^{2,p}(\Omega)} + \|\pi\|_{W^{1,p}(G)} \le Mt^{-\tilde c}e^{-ct}, \quad t>0.
\end{equation}
\end{theorem}

Our strategy to prove Theorem \ref{mainthm} may be described as follows. Recall that Proposition \ref{thm5.1} assures the unique existence of the strong solution $v$ to (2.3)--(2.4) on 
the time interval $(0, T^*]$. In the sequel, we fix some $t_1\in (0,T^*)$ and regard $v(t_1)\in D(A_p)$ as the new initial data. We hence may assume without loss of generality 
that $a\in D(A_p)$.
This will be assumed until the end of Step 7 of the proof below.

Consider first the case where  $p = 2$.
We then prove that the unique, local, strong solution $v$ constructed in $[0, T^*]$ may be extended to a strong solution on $[0, T]$ for any $T \in (T^*,\infty)$. 
In fact, the a priori estimate \eqref{a priori bound} given below yields that $\sup_{0\le t\le T}\|v(t)\|_{H^2(\Omega)}$ must be bounded by some constant $B = B(T, \|a\|_{H^2(\Omega)})$.
Propositions \ref{thm5.1} and \ref{thm5.2} enable us to choose $T_B>0$, depending only on $B$, such that $v$ may be  extended to a strong solution on $[0, T^*+T_B]$.
If $T^*+T_B < T$, then $\|v(T^*+T_B)\|_{H^2(\Omega)} \le B$, so we may extend $v$ to $[0, T^*+2T_B]$. Repeating this argument, we obtain  a unique, strong solution to (2.3)--(2.4) on $[0, T]$.
Once this fact is established, we prove the global existence for $p=2$. In the final step of the proof below, we show that $\|v(t)\|_{H^2(\Omega)}$ is globally bounded and 
is even exponentially decaying as $t\to\infty$. This property  is then extended to the case $p\ge 6/5$ by a bootstrap argument.

Before starting the proof of the a priori estimates, we observe first the following estimates concerning the two-dimensinal  Stokes equations and the three dimensional 
heat equations.

\begin{lemma} \label{lem6.1}
a) Let $f \in L^2(G)^2$ and  let $(v,\pi)$ be a solution of the equation $\partial_tv - \Delta_Hv + \nabla_H\pi = f$ satisfying $\mathrm{div}_H\,v = 0$ in $G$ and such that 
$v$ and $\pi$ are periodic on $\partial G$. Then there exists a constant $C>0$ such that 
\begin{equation*}
		8\partial_t\|\nabla_Hv\|_{L^2(G)}^2 + \|\Delta_Hv\|_{L^2(G)}^2 + \|\nabla_H\pi\|_{L^2(G)}^2 \le C\|f\|_{L^2(G)}^2, \quad f \in L^2(G)^2.
\end{equation*}
b) Let $f \in L^2(\Omega)$ and let $v$ be a solution of $\partial_tv - \Delta v = f$ in $\Omega$ such that $v_z=0$ on $\Gamma_u$, $v=0$ on $\Gamma_b$ and 
$v$ is periodic on $\Gamma_l$.
Then there exists a constant $C>0$ such that 
\begin{equation*}
		\partial_t \|\nabla v\|_{L^2(\Omega)}^2 + \|\Delta v\|_{L^2(\Omega)}^2 \le C \|f\|_{L^2(\Omega)}^2, \quad f \in L^2(\Omega).
	\end{equation*}
\end{lemma}

\begin{proof}
a) Multiplying the equation by $\partial_tv$ or $-\Delta_Hv$, integrating by parts over $G$, and adding the resulting equations, we find that
	\begin{equation*}
		\partial_t\|\nabla_Hv\|_{L^2(G)}^2 + \|\partial_tv\|_{L^2(G)}^2 + \|\Delta_Hv\|_{L^2(G)}^2 = (f,\partial_tv - \Delta_Hv).
	\end{equation*}
Note that the pressure terms give no contributions, thanks to the periodic boundary conditions. Then, evaluating the pressure term 
by $\|\nabla_H\pi\|_{L^2(G)}^2 \le 3(\|f\|_{L^2(G)}^2 + \|\partial_tv\|_{L^2(G)}^2 + \|\Delta_Hv\|_{L^2(G)}^2)$, we obtain
	\begin{align*}
		4\partial_t\|\nabla_Hv\|_{L^2(G)}^2 + \|\partial_tv\|_{L^2(G)}^2 + \|\Delta_Hv\|_{L^2(G)}^2 + \|\nabla_H\pi\|_{L^2(G)}^2 
		\le\; 3\|f\|_{L^2(G)}^2 + 4(f,\partial_tv - \Delta_Hv).
	\end{align*}
An absorbing argument gives the desired result.\\
b) This can be proved by multiplying the equation with $-\Delta v$ and integrating by parts over $\Omega$.
\end{proof}

\begin{remark} \label{rem6.1}
We note that there is a contant $C>0$ such that 
$\|\nabla_H^2v\|_{L^2(G)} \le C\|\Delta_Hv\|_{L^2(G)}$ and $\|v\|_{H^2(\Omega)} \le C\|\Delta v\|_{L^2(\Omega)}$, respectively.
\end{remark}

In addition to the vertical average $\bar v$ of $v$, which was already introduced in \eqref{eq2.02},  we now define the  fracturing part $\tilde v$ of $v$ by 
$\tilde v := v - \bar v$.  It follows from  \eqref{eq2.4} that $\|\tilde v\|_{L^q(\Omega)}\le (1+h^{-1/q})\|v\|_{L^q(\Omega)}$ for all $q\in(1,\infty)$.
We further notice that
\begin{equation*}
	\mathrm{div}_H\,\bar v=0, \quad \mathrm{div}_H\,\tilde v = \mathrm{div}_H\,v, \quad \bar w=0, \quad \tilde w=w, \quad \bar v_z=0, \quad \tilde v_z=v_z.
\end{equation*}

The following estimates will be useful later on. 

\begin{lemma} \label{lem6.2}
Let $p,q,r\in(1,\infty)$ and $g:\Omega \to \R^2$. Then \\
a) $$
\int_\Omega (\tilde v\cdot\nabla_H g + w g_z)\cdot |g|^{q-2}g  = 0 \quad \mbox{ and } \quad  \int_\Omega (\bar v\cdot\nabla_H g)\cdot |g|^{q-2}g  = 0,
$$
whenever the integrals are well defined.\\
b) Let  $\frac1p(\frac12+\frac1q)\ge\frac1r$, $g:\Omega \to \R^2$ and $z\in(-h,0)$.  Then  there exists  a constant $C>0$ such that 
$$
\| |g(\cdot,z)|^p \|_{L^q(G)} \le C (\|g(\cdot,z)\|_{L^r(G)}^p + \|g(\cdot,z)\|_{L^r(G)}^{p-r/2} \| \nabla_H|g(\cdot,z)|^{r/2} \|_{L^2(G)} ).
$$
\end{lemma}

\begin{proof}
a) The assertion follows from integration by parts. In fact, the volume integrals disappear since $\mathrm{div}\,u=0$ or $\mathrm{div}_H\,\bar v=0$. The same is true for the 
surface integrals since $w=0$ on $\Gamma_u\cup\Gamma_b$ and $\nu_{\partial\Omega}=0$ on $\Gamma_u\cup\Gamma_b$ and due to the periodic boundary conditions on $\Gamma_l$.\\
b)  For simplicity of notation, we write $g$ instead of $g(\cdot,z)$. Observe that $\| |g|^p \|_{L^q(G)} = \| |g|^{r/2} \|_{L^\alpha(G)}^\beta$, where $\alpha=2pq/r$ and $\beta=2p/r$.
On the other hand, the embedding $H^\gamma(G)\hookrightarrow L^\alpha(G)$ for $\gamma=1-2/\alpha$ together with interpolation gives 
$\|f\|_{L^\alpha(G)} \le C\|f\|_{L^2(G)}^{1-\delta} \|f\|_{H^1(G)}^\delta$ for $\delta\in[\gamma,1]$ and smooth functions  $f$. 
Hence, $\| |g|^p \|_{L^q(G)}$ is bounded by $C\||g|^{r/2}\|_{L^2(G)}^{\beta(1-\delta)}\, \||g|^{r/2}\|_{H^1(G)}^{\beta\delta}$. Next, choosing $\delta$ such that 
$\beta\delta=1$, (which is possible since $\frac1p(\frac12+\frac1q)\ge\frac1r$)  and noting that $\beta(1-\delta) = 2p/r-1$ and $\|f\|_{H^1(G)} = \|f\|_{L^2(G)} + \|\nabla_Hf\|_{L^2(G)}$, the desired estimate follows.
\end{proof}

We now give a proof of Theorem \ref{mainthm}.

\begin{proof}[Proof of Theorem \ref{mainthm}]
Multiplying \eqref{eq2.2}$_1$ by $v$, integrating over $\Omega$ and making use of Lemma \ref{lem6.2} as well of  Poincar\'e's inequality yield the existence of a constant $C>0$ satisfying 
\begin{equation} \label{eq6.1}
	\|v(t)\|_{L^2(\Omega)}^2 + \int_0^t \|v(s)\|_{H^1(\Omega)}^2\,ds \le C\|a\|_{L^2(\Omega)}^2, \quad t\in(0,\infty).
\end{equation}
Note that  $\bar v$ and $\tilde v$ also admit $L^\infty(L^2)\cap L^2(H^1)$-type bounds  as above. The energy inequality \eqref{eq6.1} will be the starting point of our proof.  

We proceed by taking  the vertical average of \eqref{eq2.2}$_1$. Following Cao and Titi \cite{CaTi07}, we obtain the following equations for $\bar v$ and $\tilde v$: 
\begin{equation} \label{eq6.2}
\begin{array}{rll}
\partial_t \bar v - \Delta_H \bar v + \nabla_H\pi &= -\bar v\cdot\nabla_H\bar v - \frac1h \int_{-h}^0 (\tilde v\cdot\nabla_H\tilde v + \mathrm{div}_H\, v\,\tilde v)\,dz - 
\frac1h v_z|_{\Gamma_b} & \text{in }\; G, \\
\mathrm{div}_H\,\bar v &= 0 & \text{in }\; G,
\end{array}
\end{equation}
and
\begin{equation}\label{eq6.3}
\partial_t\tilde v - \Delta\tilde v + \tilde v\cdot\nabla_H\tilde v + v_3v_z + \bar v\cdot\nabla_H\tilde v  = 
- \tilde v\cdot\nabla_H\bar v + \frac1h \int_{-h}^0 (\tilde v\cdot\nabla_H\tilde v + \mathrm{div}_H\,v\,\tilde v)\,dz + \frac1h v_z|_{\Gamma_b} \quad\text{in }\;\Omega. 
\end{equation}
In the sequel, we will derive bounds for ${\bar v}, v_z$ and $\tilde v$.
Note that each of the three bounds are depending on the other two and hence, each bound  given in \eqref{eq6.4}, \eqref{eq6.6} and \eqref{eq6.8} below is  not closed by itself alone.
However, let us emphasize that adding the three estimates yields an estimate, see \eqref{eq6.9} below, to which the classical Gronwall inequality is applicable directly.  
As already written in the introduction, the novelty of our approach lies in the fact that we are dealing with $L^2(L^2)$-estimates for $\nabla_H\pi$ and $L^\infty(L^4)$-estimates 
for $\tilde v$, whereas the authors in \cite{CHKTZ14, EvGa13, KuZi07, KuZi08} performed $L^2(L^{3/2})$-estimates for $\nabla_H\pi$ and $L^\infty(L^6)$-estimates for $v$. 
We subdivide our proof into six steps. 

\vspace{.2cm}\noindent
{\it Step 1 : Estimates for ${\nabla_H\bar v}\in L^\infty(L^2)$}: \\
Applying Lemma \ref{lem6.1}a) to \eqref{eq6.2}, we obtain
\begin{eqnarray*}
	8\partial_t\|\nabla_H\bar v(t)\|_{H^1(G)}^2 + \|\Delta_H\bar v\|_{H^2(G)}^2 + \|\nabla_H\pi\|_{L^2(G)}^2 
	\hspace{-.3cm} &\le& \hspace{-.3cm}   C_1(
		\big\| |\bar v||\nabla_H\bar v| \big\|_{L^2(G)}^2 + \big\| |\tilde v||\nabla_H\tilde v| \big\|_{L^2(\Omega)}^2 + \|v_z\|_{L^2(\Gamma_b)}^2) \\
        &=:& { I_1 + I_2 + I_3},
\end{eqnarray*}
where we have used the fact that $\|\bar f\|_{L^2(G)} \le C\|f\|_{L^2(\Omega)}$. Let us  estimate each term of the right hand side above seperatly. 
The interpolation inequality $\|f\|_{L^4(G)} \le C\|f\|_{L^2(G)}^{1/2}\|f\|_{H^1(G)}^{1/2}$ yields
\begin{align*}
{I_1}	\le \; &C\|\bar v\|_{L^4(G)}^2 \|\nabla_H\bar v\|_{L^4(G)}^2 \\
	\le\; &C(\|\bar v\|_{L^2(G)}^2 + \|\bar v\|_{L^2(G)}\|\nabla_H\bar v\|_{L^2(G)}) (\|\nabla_H\bar v\|_{L^2(G)}^2 + \|\nabla_H\bar v\|_{L^2(G)}\|\Delta_H\bar v\|_{L^2(G)}) \\
	=\; &C\|\bar v\|_{L^2(G)}^2\|\nabla_H\bar v\|_{L^2(G)}^2 + C\|\bar v\|_{L^2(G)}^2\|\nabla_H\bar v\|_{L^2(G)}\|\Delta_H\bar v\|_{L^2(G)} \\
	&\hspace{.3cm} + C\|\bar v\|_{L^2(G)}\|\nabla_H\bar v\|_{L^2(G)}^3 + C\|\bar v\|_{L^2(G)}\|\nabla_H\bar v\|_{L^2(G)}^2\|\Delta_H\bar v\|_{L^2(G)} \\
	\le\; &C( \|v\|_{L^2(\Omega)}^2 + \|v\|_{L^2(\Omega)}^4 )\|v\|_{H^1(\Omega)}^2 + \|\Delta_H\bar v\|_{L^2(G)}^2 \\
	&\hspace{.3cm} + C( \|v\|_{L^2(\Omega)} + \|v\|_{L^2(\Omega)}^2 )(\|v\|_{H^1(\Omega)} + \|v\|_{H^1(\Omega)}^2)\|\nabla_H\bar v\|_{L^2(G)}^2,
\end{align*}
where we have used the estimates $\|\nabla_H^2\bar v\|_{L^2(G)}\le C\|\Delta_H\bar v\|_{L^2(G)}$ (see Remark \ref{rem6.1}) in the second line
and $\|\nabla_H\bar v\|_{L^2(G)}\le C\|v\|_{H^1(\Omega)}$ in the last line.

Since $|\nabla_H\tilde v|\le |\nabla\tilde v|$, we obtain for the second term ${I_2}$ on the right hand side above that ${I_2} \leq C_1\| |\tilde v||\nabla\tilde v| \|_{L^2(\Omega)}^2$.
In view of the trace theorem, Poincar\'e's and Young's inequalities,  ${I_3} \le C\|v_z\|_{L^2(\Omega)}^2 + 1/4\|\nabla v_z\|_{L^2(\Omega)}^2$.
Consequently, there exist constants $C_1,C>0$ such that

\begin{eqnarray}\label{eq6.4}
8\partial_t\|\nabla_H\bar v\|_{H^1(G)}^2 + \|\nabla_H\pi\|_{L^2(G)}^2  &\le& C_1\big\| |\tilde v||\nabla\tilde v| \big\|_{L^2(G)}^2 
+ \frac14 \|\nabla v_z\|_{L^2(G)}^2 \nonumber \\ 
& & + C(1 + \|v\|_{L^2(\Omega)}^2 + \|v\|_{L^2(\Omega)}^4 )\|v\|_{H^1(\Omega)}^2 \\   
&& + C ( \|v\|_{L^2(\Omega)} + \|v\|_{L^2(\Omega)}^2 ) (\|v\|_{H^1(\Omega)} + \|v\|_{H^1(\Omega)}^2)\|\nabla_H\bar v\|_{L^2(G)}^2, \nonumber
\end{eqnarray}
where we made use of  the fact that $\|v_z\|_{L^2(\Omega)} \le \|v\|_{H^1(\Omega)}$.

\vspace{.1cm}\noindent
{\it Step 2: Estimates for $v_z\in L^\infty(L^2)$}:\\
We multiply \eqref{eq2.2}$_1$ by $-\partial_zv_z$, integrate over $\Omega$ and use Lemma \ref{lem6.2}. 
As discussed in \cite[p.\ 2743]{KuZi07}, all the boundary integrals vanish except the one involving $\nabla_H\pi$.
The resulting equation reads 
\begin{eqnarray}\label{eq6.5}
\frac12 \partial_t \|v_z\|_{L^2(\Omega)}^2 + \|\nabla v_z\|_{L_2(\Omega)}^2 
	&=& -\int_G \nabla_H\pi\cdot v_z|_{\Gamma_b}  - \int_\Omega (v_z\cdot\nabla_Hv)\cdot v_z  + \int_\Omega \mathrm{div}_H\, v\, v_z\cdot v_z\, \notag \\
	&=:& {I_4+I_5+I_6}. 
\end{eqnarray}
Recalling that $v = \bar v + \tilde v$, we find from further integrations by parts that 
\begin{align*}
	{I_5} &= -\int_\Omega (v_z\cdot\nabla_H\bar v)\cdot v_z - \int_\Omega (v_z\cdot\nabla_H\tilde v)\cdot v_z  \\
	&= -\int_\Omega (v_z\cdot\nabla_H\bar v)\cdot v_z  + \int_\Omega \mathrm{div}_H\,v_z\, \tilde v\cdot v_z  + \int_\Omega (v_z\cdot\nabla_Hv_z)\cdot\tilde v  \\
	&=: {I_{51} + I_{52} + I_{53}},
\end{align*}
and that 
\begin{equation*}
	{I_6} = \int_\Omega \mathrm{div}_H\,\tilde v\, v_z\cdot v_z = -2\int_\Omega (\tilde v\cdot\nabla_Hv_z)\cdot v_z.
\end{equation*}
We now estimate the above five terms.
In view of the trace theorem, Poincar\'e's and Young's inequalities, we obtain  
$$
|{ I_4}| \le \frac14\|\nabla_H\pi\|_{L^2(G)}^2 + C\|v_z\|_{L^2(\Omega)}^2 + \frac16\|\nabla v_z\|_{L^2(\Omega)}^2.
$$
Furthermore, by Fubini's theorem, H\"older's and Minkowski's inequality and Lemma \ref{lem6.2} 
\begin{eqnarray*}
	|{ I_{51}}| &\le& \int_G|\nabla_H\bar v| \int_{-h}^0|v_z|^2\,dz \le C\|\nabla_H\bar v\|_{L_2(G)} \left\| \int_{-h}^0 | v_z|^2\,dz \right\|_{L_2(G)} 
	\le C\|\nabla_H\bar v\|_{L_2(G)} \int_{-h}^0 \big\| | v_z|^2 \big\|_{L_2(G)}\,dz \\
	&\le& C\|\nabla_H\bar v\|_{L_2(G)} \int_{-h}^0 \left( \| v_z\|_{L^2(G)}^2 + \| v_z\|_{L^2(G)}\|\nabla_H v\|_{L^2(G)} \right) \,dz \\
&\le& C\|\nabla_H\bar v\|_{L_2(G)}\| v_z\|_{L^2(\Omega)}^2 + C\|\nabla_H\bar v\|_{L^2(G)}^2 \| v_z\|_{L^2(\Omega)}^2 + \frac16\big\| \nabla_H v_z \big\|_{L_2(\Omega)}^2. 
\end{eqnarray*}
For the remaining  terms we have
\begin{equation*}
	{ |I_{52}|+|I_{53}|+|I_{6}|} \le C\int_\Omega |\tilde v||\tilde v_z||\nabla_Hv_z|  \le 
C \big\| |\tilde v||\nabla\tilde v| \big\|_{L^2(\Omega)}^2 + \frac16\|\nabla v_z\|_{L^2(\Omega)}^2,
\end{equation*}
where we have used $v_z=\tilde v_z$, $|v_z|\le|\nabla v|$, and $|\nabla_Hv_z|\le|\nabla v_z|$.
Combining these estimates with \eqref{eq6.5} we conclude
\begin{eqnarray}\label{eq6.6}
\partial_t\|v_z\|_{L^2(\Omega)}^2 + \|\nabla v_z\|_{L^2(\Omega)}^2 &\le& \frac12\|\nabla_H\pi\|_{L^2(G)}^2 + C\|v\|_{H^1(\Omega)}^2 + 
2C_2\big\| |\tilde v||\nabla\tilde v| \big\|_{L^2(\Omega)}^2 \\ 
	&& + C (\|v\|_{H^1(\Omega)} + \|v\|_{H^1(\Omega)}^2)\|v_z\|_{L^2(\Omega)}^2, \notag
\end{eqnarray}
where we have used $\|v_z\|_{L^2(\Omega)}\le \|v\|_{H^1(\Omega)}$ and $\|\nabla_H\bar v\|_{L^2(G)} \le C\|v\|_{H^1(\Omega)}$.

\vspace{.2cm}\noindent
{\it Step 3: Estimates for $\tilde v \in L^\infty(L^4)$} \\
Multiplying \eqref{eq6.3} by $|\tilde v|^2\,\tilde v$, integrating by parts in $\Omega$, and using Lemma \ref{lem6.2}, we obtain 
\begin{eqnarray}\label{eq6.7}
	\frac14\partial_t\|\tilde v\|_{L^4(\Omega)}^4 &+& \frac12\big\|\nabla |\tilde v|^2\big\|_{L^2(\Omega)}^2 + \big\| |\tilde v||\nabla\tilde v| \big\|_{L^2(\Omega)}^2  \nonumber \\
	& = & -\int_\Omega (\tilde v\cdot\nabla_H\bar v) \cdot |\tilde v|^2\,\tilde v 
		+ \frac1h \int_\Omega \int_{-h}^0 (\tilde v\cdot\nabla_H\tilde v + \mathrm{div}_H\,v\,\tilde v)\,dz \cdot |\tilde v|^2\,\tilde v \\ 
            &&   + \frac1h \int_\Omega v_z|_{\Gamma_b}\cdot |\tilde v|^2\,\tilde v =: { I_7 + I_8 + I_9}. \nonumber
\end{eqnarray}
By Fubini's theorem, H\"older's and Minkowski's inequalities as well as by  Lemma \ref{lem6.2}, the term $I_7$ can be  bounded as
\begin{align*}
 { I_7} 	&\le C\int_G |\nabla_H\bar v| \int_{-h}^0 |\tilde v|^4\,dz
	\le C\|\nabla_H\bar v\|_{L_2(G)} \left\| \int_{-h}^0 |\tilde v|^4\,dz \right\|_{L_2(G)} 
	\le C\|\nabla_H\bar v\|_{L_2(G)} \int_{-h}^0 \big\||\tilde v|^4 \big\|_{L_2(G)}\,dz \\
	&\le C\|\nabla_H\bar v\|_{L_2(G)} \int_{-h}^0 \left( \|\tilde v\|_{L^4(G)}^4 + \|\tilde v\|_{L^4(G)}^2\|\nabla_H |\tilde v|^2\|_{L^2(G)} \right) \,dz  \\
	&\le C\|\nabla_H\bar v\|_{L_2(G)}\|\tilde v\|_{L^4(\Omega)}^4 + C\|\nabla_H\bar v\|_{L^2(G)}^2 \|\tilde v\|_{L^4(\Omega)}^4 + \frac16\big\|\nabla_H|\tilde v|^2 \big\|_{L_2(\Omega)}^2.
\end{align*}
In a similar manner, the term $I_8$ may be  estimated as
\begin{align*}
{ I_8}	&\le C\int_G \left( \int_{-h}^0 |\tilde v||\nabla_H\tilde v|\,dz \right) \left( \int_{-h}^0 |\tilde v|^3\,dz \right) 
	\le C\left\| \int_{-h}^0 |\tilde v||\nabla_H\tilde v|\,dz \right\|_{L^{4/3}(G)} \left\| \int_{-h}^0 |\tilde v|^3\,dz \right\|_{L^4(G)} \\
	&\le C\int_{-h}^0 \big\| |\tilde v||\nabla_H\tilde v| \big\|_{L^{4/3}(G)} \,dz \int_{-h}^0 \big\||\tilde v|^3\big\|_{L^4(G)}\,dz \\
	&\le C \left(\int_{-h}^0 \|\tilde v\|_{L^4(G)}\|\nabla_H\tilde v\|_{L^2(G)}\,dz\right) \int_{-h}^0\left( \|\tilde v\|_{L^4(G)}^3 + \|\tilde v\|_{L^4(G)}\big\|\nabla_H|\tilde v|^2\big\|_{L^2(G)} \right)\,dz  \\
	&\le C\|\tilde v\|_{L^4(\Omega)}\|\nabla_H\bar v\|_{L_2(\Omega)} (\|\tilde v\|_{L^4(\Omega)}^3 + \|\tilde v\|_{L^4(\Omega)}\big\|\nabla_H|\tilde v|^2\big\|_{L^2(\Omega)}) \\
	&\le C\|\nabla_H\bar v\|_{L_2(G)}\|\tilde v\|_{L^4(\Omega)}^4 + C\|\nabla_H\bar v\|_{L^2(G)}^2 \|\tilde v\|_{L^4(\Omega)}^4 + \frac16\big\|\, \nabla_H|\tilde v|^2 \big\|_{L_2(\Omega)}^2.
\end{align*}
Finally, by the trace theorem and by Poincar\'e's inequality as well as  Lemma \ref{lem6.2}, we estimate the term $I_9$ as
\begin{align*}
{ I_9}	&\le C\int_G |v_{z|_{\Gamma_b}}|\int_{-h}^0|\tilde v|^3\,dz  \le C\|v_z\|_{L^2(\Gamma_b)} \int_{-h}^0\big\||\tilde v|^3\big\|_{L^2(G)}\,dz \\
	&\le C\|v_z\|_{L^2(\Gamma_b)} \int_{-h}^0( \|\tilde v\|_{L^4(G)}^3 + \|\tilde v\|_{L^4(G)}\big\|\nabla_H|\tilde v|^2\big\|_{L^2(G)} )\,dz \\
	&\le C\|v_z\|_{L^2(\Gamma_b)}\|\tilde v\|_{L^4(\Omega)}^3 + C\|v_z\|_{L^2(\Gamma_b)}\|\tilde v\|_{L^4(\Omega)}\big\|\nabla_H|\tilde v|^2\big\|_{L^2(\Omega)} \\
	&\le C\|v_z\|_{L^2(\Omega)}^{1/2} \|\nabla v_z\|_{L^2(\Omega)}^{1/2}\|\tilde v\|_{L^4(\Omega)}^3
		+ C\|v_z\|_{L^2(\Omega)}^{1/2} \|\nabla v_z\|_{L^2(\Omega)}^{1/2}\|\tilde v\|_{L^4(\Omega)}\big\|\nabla_H|\tilde v|^2\big\|_{L^2(\Omega)}  \\
	&\le C\|v_z\|_{L^2(\Omega)}^{2/3}\|\tilde v\|_{L^4(\Omega)}^4 + C\|v_z\|_{L^2(\Omega)}^2\|\tilde v\|_{L^4(\Omega)}^4 + \frac1{4C_3}\|\nabla v_z\|_{L^2(\Omega)}^2 + \frac16\big\|\, 
\nabla_H|\tilde v|^2 \big\|_{L_2(\Omega)}^2, 
\end{align*}
where $C_3:=2(C_1+2C_2)$ is determined by $C_1$ and $C_2$ defined as in Step 1  and 2, respectively. Combining these estimates with \eqref{eq6.7} and multiplying with $C_3$, we conclude
that 
\begin{align}\label{eq6.8}
	\frac{C_3}{4} \partial_t\|\tilde v\|_{L^4(\Omega)}^4 &+ C_3 \big\||\tilde v||\nabla\tilde v|\big\|_{L^2(\Omega)}^2 
	\le  C(\|v\|_{H^1(\Omega)}^{2/3} + \|v\|_{H^1(\Omega)} + \|v\|_{H^1(\Omega)}^2)\|\tilde v\|_{L^4(\Omega)}^4 + \frac14 \|\nabla v_z\|_{L^2(\Omega)}^2, 
\end{align}
where we made use of the estimates  $\|\nabla_H\bar v\|_{L^2(G)}\le C\|v\|_{H^1(\Omega)}$ and $\|v_z\|_{L^2(\Omega)}\le \|v\|_{H^1(\Omega)}$.

\vspace{.2cm} \noindent
{\it Step 4: Adding the above estimates} \\
Addition of \eqref{eq6.4}, \eqref{eq6.6} and \eqref{eq6.8} enables us to absorb the terms 
$\big\| |\tilde v||\nabla\tilde v| \big\|_{L^2(\Omega)}^2$, $\|\nabla v_z\|_{L^2(\Omega)}^2$ and $\|\nabla_H\pi\|_{L^2(G)}^2$ into the left hand side. This leads us to 
\begin{eqnarray}\label{eq6.9}
 \partial_t \big(8\|\nabla_H\bar v\|_{L^2(G)}^2 &+& \|v_z\|_{L^2(\Omega)}^2 + \frac{C_3}{4}\|\tilde v\|_{L^4(\Omega)}^4\big)  
	+ \frac12 \big(\|\nabla_H\pi\|_{L^2(G)}^2 + \|\nabla v_z\|_{L^2(\Omega)}^2 + C_3\big\| |\tilde v||\nabla_H\tilde v| \big\|_{L^2(\Omega)}^2\big)  \nonumber \\
&&	\le\; K_1(t)\big(8\|\nabla_H\bar v\|_{L^2(G)}^2 + \|v_z\|_{L^2(\Omega)}^2 + (C_3/4)\|\tilde v\|_{L^4(\Omega)}^4\big) + K_2(t), 
\end{eqnarray}
where $K_1(t) := C(1 + \|v(t)\|_{L^2(\Omega)} + \|v(t)\|_{L^2(\Omega)}^2 )(\|v(t)\|_{H^1(\Omega)}^{2/3} + \|v(t)\|_{H^1(\Omega)} + \|v(t)\|_{H^1(\Omega)}^2)$
and $K_2(t) := C(1 + \|v(t)\|_{L^2(\Omega)}^2 + \|v(t)\|_{L^2(\Omega)}^4) \|v(t)\|_{H^1(\Omega)}^2$.
It follows from \eqref{eq6.1} and H\"older's inequality that
\begin{align*}
	\int_0^t K_1(s)\,ds &\le C(1 + \|a\|_{L^2(\Omega)} + \|a\|_{L^2(\Omega)}^2 )( \|a\|_{L^2(\Omega)}^{2/3}t^{2/3} + \|a\|_{L^2(\Omega)}t^{1/2} + \|a\|_{L^2(\Omega)}^2) < \infty, \quad t\ge 0, \\
	\int_0^t K_2(s)\,ds &\le C(1 + \|a\|_{L^2(\Omega)}^2 + \|a\|_{L^2(\Omega)}^4) \|a\|_{L^2(\Omega)}^2 < \infty, \quad t\ge 0.
\end{align*}
Applying Gronwall's inequality thus yields
\begin{align}
	8\|\nabla_H\bar v\|_{L^2(G)}^2 + \|v_z\|_{L^2(\Omega)}^2 + (C_3/4)\|\tilde v\|_{L^4(\Omega)}^4 &\le \left( C(\|a\|_{H^1(\Omega)}^2 + \|a\|_{H^1(\Omega)}^4) + \int_0^t K_2(s)\,ds \right) e^{\int_0^t K_1(s)\,ds} \notag \\
	&=: B_1(t, \|a\|_{H^1(\Omega)}),  \quad t\ge0,\label{eq6.10}
\end{align}
where we estimated the left hand side  of \eqref{eq6.10} for $t=0$ by $C(\|a\|_{H^1(\Omega)}^2 + \|a\|_{H^1(\Omega)}^4)$.
By \eqref{eq6.1}, we obtain further  that $\|\bar v(t)\|_{H^1(G)}^2 \le C\|a\|_{L^2(\Omega)}^2 +  \frac18 B_1(t, \|a\|_{H^1(\Omega)})$  and that
\begin{align} \label{eq6.11}
	\frac12 \int_0^t( \|\nabla_H\pi\|_{L^2(G)}^2 + \|\nabla v_z\|_{L^2(\Omega)}^2 + C_3\big\| |\tilde v||\nabla\tilde v| \big\|_{L^2(\Omega)}^2)\,ds \le B_1(t, \|a\|_{H^1(\Omega)}), \quad t>0.
\end{align}

\vspace{.2cm}\noindent
{\it Step 5: Estimates for $v\in L^\infty(H^1)$} \\
We consider  \eqref{eq2.2}$_1$ as an inhomogeneous  heat equation of the form $\partial_tv - \Delta v = - v\cdot\nabla_Hv - wv_z - \nabla_H\pi$
and apply Lemma  \ref{lem6.1}b). Since $v\cdot\nabla_Hv = \bar v\cdot\nabla_H\bar v + \bar v\cdot\nabla_H\tilde v + \tilde v\cdot\nabla_H\bar v + \tilde v\cdot\nabla_H\tilde v$, we obtain 
\begin{align}\label{eq6.12}
	&\partial_t\|\nabla v\|_{L^2(\Omega)}^2 + \|\Delta v\|_{L^2(\Omega)}^2 \le C(  \|\bar v\cdot\nabla_H\bar v\|_{L^2(G)}^2 + \|\bar v\cdot\nabla_H\tilde v\|_{L^2(\Omega)}^2 + \|\tilde v\cdot\nabla_H\bar v\|_{L^2(\Omega)}^2 \notag \\
	&\hspace{4.8cm} + \|wv_z\|_{L^2(\Omega)}^2 + \|\tilde v\cdot\nabla_H\tilde v\|_{L^2(\Omega)}^2 + \|\nabla_H\pi\|_{L^2(G)}^2 ). 
\end{align}

We estimate each term on the right hand side of  \eqref{eq6.12}.  The last two terms were  already  estimated in  \eqref{eq6.11}.
Next,  the interpolation inequality $\|f\|_{L^4(G)} \le C\|f\|_{L^2(G)}^{1/2}\|f\|_{H^1(G)}^{1/2}$ yields
\begin{align*}
	\|\bar v\cdot\nabla_H\bar v\|_{L^2(G)}^2 &\le C\|\bar v\|_{L^4(G)}^2\|\bar v\|_{H^1(G)}\|\bar v\|_{H^2(G)} \le C\|\bar v\|_{L^4(G)}^2\|v\|_{H^1(\Omega)}\|v\|_{H^2(\Omega)} \\
		&\le C\|\bar v\|_{H^1(G)}^4\|v\|_{H^1(\Omega)}^2 + 1/8 \|\Delta v\|_{L^2(\Omega)}^2,
\end{align*}
where we have used $\|v\|_{H^2(\Omega)}\le C\|\Delta v\|_{L^2(\Omega)}$; see Remark \ref{rem6.1}.

The second term above is again estimated by interpolation as
\begin{align*}
	\|\bar v\cdot\nabla_H\tilde v\|_{L^2(\Omega)}^2 &\le C\|\bar v\|_{L^6(G)}^2\|\tilde v\|_{W^{1,3}(\Omega)}^2 \le C\|\bar v\|_{L^6(G)}^2\|v\|_{W^{1,3}(\Omega)}^2 
	\le C\|\bar v\|_{H^1(G)}^2 \|v\|_{H^1(\Omega)} \|v\|_{H^2(\Omega)} \\
	&\le C\|\bar v\|_{H^1(G)}^4 \|v\|_{H^1(\Omega)}^2 + 1/8 \|\Delta v\|_{L^2(\Omega)}^2,
\end{align*}
where we used the embedding $H^{3/2}(\Omega)\hookrightarrow W^{1,3}(\Omega)$. 

Similarly, since $H^{3/2}(G)\hookrightarrow W^{1,4}(G)$, the third term above is bounded by
\begin{align*}
	\|\tilde v\cdot\nabla_H\bar v\|_{L^2(\Omega)}^2 &\le C\|\tilde v\|_{L^4(\Omega)}^2\|\bar v\|_{W^{1,4}(G)}^2 \le C\|\tilde v\|_{L^4(\Omega)}^2\|\bar v\|_{H^1(G)}\|\bar v\|_{H^2(G)} 
		\le C\|\tilde v\|_{L^4(\Omega)}^2\|v\|_{H^1(\Omega)}\|v\|_{H^2(\Omega)} \\
		&\le C\|\tilde v\|_{L^4(\Omega)}^4\|v\|_{H^1(\Omega)}^2 + 1/8 \|\Delta v\|_{L^2(\Omega)}^2.
\end{align*}
Finally, for the fourth term we use the  anisotropic estimate given in the proof of Lemma \ref{lem5.1}. This combined with the interpolation 
inequality $\|f\|_{L^4(G)} \le C\|f\|_{L^2(G)}^{1/2}\|f\|_{H^1(G)}^{1/2}$ and with Poincar\'e's inequality yields 
\begin{align*}
	\|wv_z\|_{L^2(\Omega)}^2 &\le C\|w\|_{L^\infty_zL^4_{xy}}^2\|v_z\|_{L^2_zL^4_{xy}}^2 \le C\|\nabla_Hv\|_{L^2_zL^4_{xy}}^2\|v_z\|_{L^2_zL^4_{xy}}^2 \\
		&\le C\|\nabla_Hv\|_{L^2_zL^2_{xy}}\|\nabla_Hv\|_{L^2_zH^1_{xy}} \|v_z\|_{L^2_zL^2_{xy}}\|v_z\|_{L^2_zH^1_{xy}} \\
		&\le C  \|\nabla v\|_{L^2(\Omega)} \|v\|_{H^2(\Omega)} \|v_z\|_{L^2(\Omega)}\|\nabla v_z\|_{L^2(\Omega)} \\
		&\le C\|v_z\|_{L^2(\Omega)}^2\|\nabla v_z\|_{L^2(\Omega)}^2  \|\nabla v\|_{L^2(\Omega)}^2 + 1/8 \|\Delta v\|_{L^2(\Omega)}^2.
\end{align*}
Combining these estimates with \eqref{eq6.12} we arrive at
\begin{equation*}
	\partial_t\|\nabla v\|_{L^2(\Omega)}^2 +1/2 \|\Delta v\|_{L^2(\Omega)}^2 \le C\|v_z\|_{L^2(\Omega)}^2\|\nabla v_z\|_{L^2(\Omega)}^2\|\nabla v\|_{L^2(\Omega)}^2 + C(\|\bar v\|_{H^1(G)}^2 + \|\tilde v\|_{L^4(\Omega)}^4)\|v\|_{H^1(\Omega)}^2.
\end{equation*}
Since $\int_0^t C\|v_z\|_{L^2(\Omega)}^2\|\nabla v_z\|_{L^2(\Omega)}^2\,ds$ and since $\int_0^t C(\|\bar v\|_{H^1(G)}^2 + \|\tilde v\|_{L^4(\Omega)}^4)\|v\|_{H^1(\Omega)}^2\,ds$ are bounded by 
$CB_1(t, \|a\|_{H^1(\Omega)})^2$ and by $C (B_1(t, \|a\|_{H^1(\Omega)})+\|a\|_{L^2(\Omega)}^2)\|a\|_{L^2(\Omega)}^2 =: L(t)$ respectively, it follows from Gronwall's inequality that
\begin{align*}
	\|\nabla v(t)\|_{L^2(\Omega)}^2 +\frac12 \int_0^t \|\Delta v\|_{L^2(\Omega)}^2\, ds &\le (\|\nabla a\|_{L^2(\Omega)}^2 + L(t)) e^{CB_1(t, \|a\|_{H^1(\Omega)})^2} < \infty, \quad t \geq 0.
\end{align*}
Consequently, $\|v(t)\|_{H^1(\Omega)} \le B_2(t, \|a\|_{H^1(\Omega)})$ for some function $B_2$.

\vspace{.2cm}\noindent
{\it Step 6: Estimates for $v_z\in L^\infty(L^3)$}: \\
We multiply $(2.3)_1$ by $-\partial_z(|v_z|v_z)$ and integrate over $\Omega$.
By an argument similar to the one that derived (6.5), we obtain
\begin{align*}
	\frac13 \partial_t \|v_z\|_{L^3(\Omega)}^3 + \frac49 \big\| \nabla |v_z|^{3/2} \big\|_{L^2(\Omega)}^2 + \big\| |v_z|^{1/2} |\nabla v_z| \big\|_{L^2(\Omega)}^2 &= -\int_{\Gamma_b} \nabla_H\pi \cdot |v_z|v_z - \int_\Omega (v_z\cdot\nabla_H)v \cdot |v_z|v_z \\
	&\hspace{3.25cm} + \int_\Omega \mathrm{div}_H\,v |v_z|^3 \\
	&=: J_1 + J_2 + J_3.
\end{align*}
We estimate each term on the right-hand side.
It follows that
\begin{align*}
	|J_2| + |J_3| &\le C\int_\Omega \|\nabla_H v\| |v_z|^3 \le C\|\nabla_H v\|_{L^2(\Omega)} \big\| |v_z|^3 \big\|_{L^2(\Omega)} = C \|\nabla_H v\|_{L^2(\Omega)} \big\| |v_z|^{3/2} \big\|_{L^4(\Omega)}^2 \\
		&\le C\|\nabla_H v\|_{L^2(\Omega)} \big\| |v_z|^{3/2} \big\|_{L^2(\Omega)}^{1/2} \big\| \nabla|v_z|^{3/2} \big\|_{L^2(\Omega)}^{3/2} \\
		&\le C\|\nabla_H v\|_{L^2(\Omega)}^4 \big\| |v_z|^{3/2} \big\|_{L^2(\Omega)}^2 + \frac19 \big\| \nabla|v_z|^{3/2} \big\|_{L^2(\Omega)}^2 \\
		&= C\|\nabla_H v\|_{L^2(\Omega)}^4 \|v_z\|_{L^3(\Omega)}^3 + \frac19 \big\| \nabla|v_z|^{3/2} \big\|_{L^2(\Omega)}^2,
\end{align*}
where we have used the interpolation inequality $\|f\|_{L^4(\Omega)} \le C\|f\|_{L^2(\Omega)}^{1/4} \|\nabla f\|_{L^2(\Omega)}^{3/4}$, with $f=0$ on $\Gamma_u$, in the second line.
For $J_1$ we have
\begin{align*}
	|J_1| &\le C\|\nabla_H \pi\|_{L^2(G)} \big\| |v_z|^2 \big\|_{L^2(\Gamma_b)} = C\|\nabla_H \pi\|_{L^2(G)} \big\| |v_z|^{3/2} \big\|_{L^{8/3}(\Gamma_b)}^{4/3} \\
		&\le C\|\nabla_H \pi\|_{L^2(G)} \big\| |v_z|^{3/2} \big\|_{L^2(\Omega)}^{1/3} \big\| \nabla|v_z|^{3/2} \big\|_{L^2(\Omega)} \\
		&\le C\|\nabla_H \pi\|_{L^2(G)}^2 \big\| |v_z|^{3/2} \big\|_{L^2(\Omega)}^{2/3} + \frac19 \big\| \nabla|v_z|^{3/2} \big\|_{L^2(\Omega)}^2
		= C\|\nabla_H \pi\|_{L^2(G)}^2 \|v_z\|_{L^3(\Omega)} + \frac19 \big\| \nabla|v_z|^{3/2} \big\|_{L^2(\Omega)}^2,
\end{align*}
where we have used the fact that the trace operator is bounded from $H^{3/4}(\Omega)$ into $L^{8/3}(\Gamma_b)$, as well as the interpolation inequality $\|f\|_{H^{3/4}(\Omega)} \le C\|f\|_{L^2(\Omega)}^{1/4} \|\nabla f\|_{L^2(\Omega)}^{3/4}$, with $f = 0$ on $\Gamma_u$.

Collecting the above estimates, we obtain
\begin{equation*}
	\partial_t \|v_z\|_{L^3(\Omega)}^3 \le C\|\nabla_H \pi\|_{L^2(G)}^2 \|v_z\|_{L^3(\Omega)} + C\|\nabla_H v\|_{L^2(\Omega)}^4 \|v_z\|_{L^3(\Omega)}^3, \quad t>0.
\end{equation*}
Dividing the both sides by $\|v_z\|_{L^3(\Omega)}$ and applying Gronwall's inequality yield
\begin{align*}
	\|v_z(t)\|_{L^3(\Omega)}^2 &\le  \left( \|v_z(0)\|_{L^3(\Omega)}^2 + \int_0^t C\|\nabla_H \pi\|_{L^2(G)}^2\,ds \right) e^{\int_0^t C\|\nabla_H v\|_{L^2(\Omega)}^4\, ds}, \quad t\ge0.
\end{align*}
Therefore, $\|v_z(t)\|_{L^3(\Omega)} \le B_3(t, \|a\|_{H^2(\Omega)})$ for some function $B_3$.

\vspace{.2cm}\noindent
{\it Step 7: Estimates for $\partial_t v\in L^\infty(L^2)$} \\
Applying the hydrostatic Helmholtz projector $P_2$ to (2.3)$_1$ we obtain 
\begin{equation}\label{eq:step7}
	\partial_t v + P_2(v\cdot\nabla_H v + w\partial_z v) + A_2v = 0, \quad t\ge0.
\end{equation}
For fixed $\tau >0$ we now set $s_\tau := v(t+\tau)$ and $D_\tau v := \frac{1}{\tau}(s_\tau v - v)$. Then, taking the difference quotient of \eqref{eq:step7} yields 
\begin{equation*} 
\partial_t D_\tau v  + P_2(s_\tau v \cdot \nabla_hD_\tau v + s_\tau w \partial_z D_\tau v) + A_2D_\tau v = -P_2(D_\tau v \cdot \nabla_H v + D_\tau w \partial_z v), \quad t\geq 0.    
\end{equation*}
Multiplying this equation by $D_\tau v = P_2 D_\tau v$ and integrating over $\Omega$ yields 
\begin{equation*}
\int_\Omega (s_\tau v \cdot \nabla_H D_\tau v + s_\tau w \partial_z D_\tau v) \cdot D_\tau v = 0, 
\end{equation*}
since $\mbox{div }s_\tau u =0$, $s_\tau w = 0$ on $\Gamma_u \cup \Gamma_b$ and $s_\tau u$ is periodic on $\Gamma_l$.   
Hence, 
\begin{equation*}
\frac12 \partial_t \|D_\tau v \|^2_{L^2(\Omega)} + \|\nabla D_\tau v\|^2_{L^2(\Omega)^2} = - \int_\Omega (D_\tau v \cdot \nabla_H)v \cdot D_\tau v - \int_\Omega D_\tau w \partial_z w 
\cdot D_\tau v =: J_4 + J_5.
\end{equation*}
For $J_4$, it follows that
\begin{align*}
	|J_4| &\le \|\nabla_H v\|_{L^2(\Omega)} \|D_\tau v\|_{L^4(\Omega)}^2 \le C\|\nabla_H v\|_{L^2(\Omega)} \|D_\tau v\|_{L^2(\Omega)}^{1/2} \|\nabla D_\tau v\|_{L^2(\Omega)}^{3/2} \\
	&\le C\|\nabla_H v\|_{L^2(\Omega)}^4 \|D_\tau v\|_{L^2(\Omega)}^2 + \frac14 \|\nabla D_\tau v\|_{L^2(\Omega)}^2,
\end{align*}
where we used interpolation  and the fact that $D_\tau v = 0$ on $\Gamma_b$. Next, we exploit the anisotropic estimates to bound $J_5$ as
\begin{align*}
	|J_5| &\le C\|D_\tau w\|_{L^\infty_zL^2_{xy}} \|v_z\|_{L^3_zL^3_{xy}} \|D_\tau v\|_{L^2_zL^6_{xy}} \le C\|\mathrm{div}_H\, D_\tau v\|_{L^2_zL^2_{xy}} \|v_z\|_{L^3(\Omega)} 
\|D_\tau v\|_{L^2_zH^{2/3}_{xy}} \\
		&\le C\|\nabla D_\tau v\|_{L^2(\Omega)} \|v_z\|_{L^3(\Omega)} \|D_\tau v\|_{L^2(\Omega)}^{1/3} \|\nabla D_\tau v\|_{L^2(\Omega)}^{2/3}
		\le C\|v_z\|_{L^3(\Omega)}^6 \|D_\tau v\|_{L^2(\Omega)}^2 + \frac14 \|\nabla D_\tau v\|_{L^2(\Omega)}^2.
\end{align*}
Consequently,
\begin{equation*}
	\partial_t \|D_\tau v\|_{L^2(\Omega)}^2 \le C(\|\nabla_H v\|_{L^2(\Omega)}^4 + \|v_z\|_{L^3(\Omega)}^6) \|D_\tau v\|_{L^2(\Omega)}^2, \quad t>0,
\end{equation*}
which leads to
\begin{equation*}
	\|D_\tau v(t)\|_{L^2(\Omega)}^2 \le \|D_\tau v (0)\|_{L^2(\Omega)}^2 e^{\int_0^t C(\|\nabla_H v\|_{L^2(\Omega)}^4 + \|v_z\|_{L^3(\Omega)}^6)\,ds}, \quad t\ge0,
\end{equation*}
as a result of Gronwall's inequality. Letting $\tau \to 0$ implies 
\begin{equation*}
	\|\partial_t v(t)\|_{L^2(\Omega)}^2 \le \|\partial_t v(0)\|_{L^2(\Omega)}^2 e^{\int_0^t C(\|\nabla_H v\|_{L^2(\Omega)}^4 + \|v_z\|_{L^3(\Omega)}^6)\,ds}, \quad t\ge0.
\end{equation*}
It remains to bound $\|\partial_t v(0)\|_{L^2(\Omega)}$. To this end, we multiply (\ref{eq:step7}) by $\partial_t v$ and let $t=0$ to obtain
\begin{equation*}
	\|\partial_t v(0)\|_{L^2(\Omega)}^2 \le (\|a\cdot\nabla_H a\|_{L^2(\Omega)} + \|\textstyle\int_{-h}^z \mathrm{div}_Ha\,d\zeta\, \partial_z a\|_{L^2(\Omega)} 
+ \|\Delta a\|_{L^2(\Omega)}) \|\partial_t v(0)\|_{L^2(\Omega)}.
\end{equation*}
Thus $\|\partial_t v(0)\|_{L^2(\Omega)} \le C(\|a\|_{H^2(\Omega)}^2 + \|a\|_{H^2(\Omega)})$  and we conclude that $\|\partial_t v(t)\|_{L^2(\Omega)} \le B_4(t, \|a\|_{H^2(\Omega)})$ for $t\ge0$ 
and some function $B_4$. 

\vspace{.2cm}\noindent
{\it Step 8: Estimates for $v\in L^\infty(H^2)$ and the proof of the Global Existence for $p=2$} \\
In view of the fact that $\|v\|_{H^2(\Omega)} \le \|Av\|_{L^2(\Omega)}$ (recall Theorem 3.1 with $\lambda=0$), it suffices for us to estimate $\|Av(t)\|_{L^2(\Omega)}$.
It follows from multiplying equation \eqref{eq:step7} by $Av$ and integrating over $\Omega$ that
\begin{equation*}
	\|Av\|_{L^2(\Omega)} \le C(\|\partial_t v\|_{L^2(\Omega)} + \|v\cdot\nabla_Hv\|_{L^2(\Omega)} + \|w\partial_zv\|_{L^2(\Omega)}), \quad t\ge0.
\end{equation*}
We see that
\begin{align*}
	\|v\cdot\nabla_Hv\|_{L^2(\Omega)} \le C\|v\|_{L^6(\Omega)} \|v\|_{W^{1,3}(\Omega)} \le C\|v\|_{H^1(\Omega)}^{3/2} \|v\|_{H^2(\Omega)}^{1/2} \le C\|v\|_{H^1(\Omega)}^3 + \frac14\|Av\|_{L^2(\Omega)},
\end{align*}
where we have used $H^1(\Omega)\hookrightarrow L^6(\Omega)$ and $\|v\|_{W^{1,3}(\Omega)} \le C\|v\|_{H^1(\Omega)}^{1/2} \|v\|_{H^2(\Omega)}^{1/2}$.
By the anisotropic estimates, we have
\begin{align*}
	\|w\partial_zv\|_{L^2(\Omega)} &\le C\|w\|_{L^6(\Omega)}\|v_z\|_{L^3(\Omega)} \le C\|\mathrm{div}_H\,v\|_{L^2_zL^6_{xy}} \|v_z\|_{L^3(\Omega)} \le C\|\mathrm{div}_H\,v\|_{L^2_zH^{2/3}_{xy}} \|v_z\|_{L^3(\Omega)} \\
		&\le C\|v\|_{H^1(\Omega)}^{1/3} \|v\|_{H^2(\Omega)}^{2/3} \|v_z\|_{L^3(\Omega)} \le C\|v\|_{H^1(\Omega)} \|v_z\|_{L^3(\Omega)}^3 + \frac14 \|v\|_{H^2(\Omega)}.
\end{align*}
Consequently,
\begin{equation} \label{a priori bound}
	\|Av\|_{L^2(\Omega)} \le C(\|\partial_tv\|_{L^2(\Omega)} + \|v\|_{H^1(\Omega)}^3 + \|v\|_{H^1(\Omega)} \|v_z\|_{L^3(\Omega)}^3), \quad t\ge0,
\end{equation}
which implies $\|v(t)\|_{H^2(\Omega)} \le B_5(t, \|a\|_{H^2(\Omega)})$.
This completes the proof of the $H^2$-a priori estimates.

Let us remark that up to this point we considered  $v(t_1)\in D(A_2)$ as initial data. Returning  to the setting where the initial data $a \in V_{1/2,2}$ and following our remarks 
made at the  beginning of this section, we proved that the local strong solution given in Propositions \ref{thm5.1} and \ref{thm5.2} extends to $(0,T]$ for any $T>0$.

\vspace{.2cm}\noindent
{\it Step 9: Global Existence for $p \in [6/5,\infty)$  and Exponential Decay} \\
Let $v\in C^1((0,T^*]; X_p)\cap C((0,T^*]; D(A_p))$ be the local solution to equation \eqref{eq5.3} corresponding to the initial value $a\in V_{1/p, p}$ constructed in Proposition \ref{thm5.1}. 
For fixed $t_0\in(0,T^*]$ we now  regard $v(t_0)$ as a new  initial value. By the embedding $W^{2,p}(\Omega)\hookrightarrow H^1(\Omega)$ which is true provided $p\ge6/5$, we see 
that $v(t_0)\in V_{1/2,2}$. The latter space  was characterized as a subspace of $H^1(\Omega)^2$ in Remark \ref{rem5.1}b).
The above steps 1--8 imply  the global existence of $v$ within the $L^2$-framework. 

In order to show the global well-posedness for $p \geq 6/5$, we  establish first the exponential decay of $v$ for $p=2$.
Recalling the proof of Proposition \ref{thm5.1}, we know that $v$ is obtained as the limit of the sequence $(v_m)$, where  $v_m$ is given by 
\begin{equation*}
	v_0(t) = e^{-tA_2}v(t_0), \quad v_{m+1}(t) = v_0(t) + \int_{t_0}^t e^{-(t-s)A_2}F_2v_m(s)\,ds.
\end{equation*}
Redefining $k_m(t) := e^{\beta t}\sup_{0\le s\le t}\|v_m(t)\|_{V_{3/4,2}}$, where $\beta$ is as in Proposition  \ref{propstokesanalytic}, the arguments given in the proof of 
Proposition \ref{thm5.1} are still valid and, in particular, if $\|{v}(t_0)\|_{H^1(\Omega)}$ is sufficiently small, 
then there exists a constant $C>0$ such that $k_m(t)\le C$ for all $m \in \N$ and all $t\ge t_0$.
This property is indeed satisfied by choosing $t_0$ suitably, since  $\inf_{t\ge0}\|v(t)\|_{H^1(\Omega)} = 0$. The latter fact follows from the observation  that the mapping 
$t \mapsto \|v(t)\|_{H^1(\Omega)}^2$ is continuous and integrable in $[0,\infty)$ by \eqref{eq6.1}.
Consequently, $\|v_m(t)\|_{V_{3/4,2}} \le Ce^{-\beta t}$ for all $t \ge t_0$ and where $C = C(t_0)$.
Now, letting $m\to\infty$ yields  $\|v(t)\|_{V_{3/4,2}} \le Ce^{-\beta t}$ for all $t\ge t_0$. By Lemmas \ref{lem5.1} and \ref{lem5.4}, we see that 
 $\|F_2v\|_{X_2}\le Ce^{-\beta t}$ and $F_2v\in C^{1/4-\epsilon}([t_0,\infty); X_2)$ with $\epsilon>0$ arbitrarily small. Moreover, by Lemmas \ref{lem5.2} and \ref{lem5.4} we see that 
$\|\partial_tv\|_{X_2} + \|v\|_{D(A_2)} \le Ce^{-ct}$ for all $t\geq t_0$ and some $C,c>0$. This implies $\|\nabla_H\pi\|_{L^2({G})} \le Ce^{-ct}$ for all $t\geq t_0$.
These estimates combined with the ones on the finite interval $(0,t_0]$ conclude \eqref{eq:exponential decay} for $p=2$.
By Proposition \ref{thm5.2},  we also have that $v\in C^{\mu}((0,\infty); D(A_2))$.

In a second step we extend the above result to $p\le 3$ by exploiting a bootstrap argument.
We may regard $v$ as the solution of the linear primitive equations with the inhomogeneous external force $f := -v\cdot\nabla_Hv - w\partial_zv$, i.e.,
\begin{equation*}
	v(t) = e^{-tA_p}a + \int_{0}^t e^{-(t-s)A_p} P_pf(s)\,ds, \qquad t>0.
\end{equation*}
Observe  that
\begin{align*}
	\|f\|_{L^3(\Omega)} \le C\|v\|_{L^6(\Omega)} \|\nabla_Hv\|_{L^6(\Omega)} + C\|\mathrm{div}_H v\|_{L^6(\Omega)} \|\partial_zv\|_{L^6(\Omega)} \le C\|v\|_{H^2(\Omega)}^2
\end{align*}
is exponentially decaying as $t\to\infty$ and that $f\in C^{\mu}((0,\infty); L^3(\Omega)^2)$.
Therefore, \eqref{eq:exponential decay} holds for $p\le3$ by Lemma \ref{lem5.4}. We also have $v\in C^{\mu}((0,\infty); D(A_3))$ by Remark \ref{rem:maximal Holder regularity}.
We now  repeat the above argument  once more, however, now taking the case $p=3$ for granted.
Combining the embedding $W^{2,3}(\Omega)\hookrightarrow W^{1,2p}(\Omega)$ for all $p\in(3,\infty)$ with the estimate $\|f\|_{L^p(\Omega)} \le C\|v\|_{W^{1,2p}(\Omega)}^2$, we see 
that $\|f(t)\|_{L^p(\Omega)}$ is exponentially decaying and H\"older continuous on $(0, \infty)$. This observation combined with Lemma \ref{lem5.4} leads to \eqref{eq:exponential decay}.
The proof of Theorem \ref{mainthm} is complete.
\end{proof}

\noindent
{\bf Acknowledgement.} The authors would like to thank T. Hishida, J. Li and E. Titi for fruitful discussions on the primitive equations and the role of $H^1$- a priori estimates
within this framework.  
The second author also would like to thank the members of DFG International Research Training Group IRTG 1529 on Mathematical Fluid Dynamics for their kind hospitality during his 
stay at TU Darmstadt.


\end{document}